\documentclass[leqno]{scrartcl}

\scrollmode
\usepackage{amsmath, amsthm, amssymb,amsfonts,mathtools}
\usepackage{verbatim,geometry,graphicx,multirow,array}
\usepackage{algpseudocode}
\usepackage{enumerate}
\usepackage{xcolor}
\usepackage[normalem]{ulem}
\usepackage{xy,psfrag}
\usepackage{datetime}
\usepackage{eucal}
\usepackage{amsmath,amsthm,amssymb,amsfonts,xcolor}
\usepackage{graphicx, caption, subcaption, listings, float}
\usepackage{multirow, tabularx}
\usepackage{array,algorithm, algpseudocode}
\usepackage{makecell, setspace}
\usepackage{color}
\usepackage{svg}
\usepackage{dirtytalk}
\usepackage[normalem]{ulem}
\usepackage{mathrsfs}

\usepackage{todonotes}

\definecolor{PrussianBlue}{rgb}{0.0, 0.19, 0.33}
\definecolor{EmeraldGreen}{rgb}{0.180, 0.545, 0.341}
\definecolor{Rust}{rgb}{0.8, 0.33, 0.0}

\usepackage[backend=biber,style=alphabetic,url=false,giveninits=true,sorting=anyt, maxbibnames=9, minbibnames=1]{biblatex} %new syntax
\DeclarePairedDelimiter{\abs}{\lvert}{\rvert}

\newcommand{\numberset}{\mathbb}
\newcommand{\N}{\numberset{N}}
\newcommand{\R}{\numberset{R}}

\renewcommand{\div}{\textup{div}}

\newcommand{\B}{\textup{B}}
\newcommand{\supt}{\textup{supp}}
\newcommand{\de}{\mathrm{d}}

\theoremstyle{plain}
\newtheorem{theorem}{Theorem}[section]
\newtheorem{lemma}[theorem]{Lemma}
\newtheorem{corollary}[theorem]{Corollary}
\newtheorem{proposition}[theorem]{Proposition}

\theoremstyle{definition}
\newtheorem{remark}[theorem]{Remark}

\newtheorem{definition}[theorem]{Definition}
\newtheorem{example}[theorem]{Example}

\newtheorem{assumptions}[theorem]{Assumptions}

\newcommand{\fr}{\penalty-20\null\hfill$\blacklozenge$}                      %quadratino nero alla fine del remark, se non vi piace, la cosa migliore e' `svuotare' la macro, cosi' non bisogna intervenire sul testo

\usepackage[colorlinks]{hyperref}
\hypersetup{hidelinks, colorlinks=true, linkcolor=Rust, urlcolor=PrussianBlue, citecolor=EmeraldGreen, linktoc=page}

\newcommand*\samethanks[1][\value{footnote}]{\footnotemark[#1]}
\newcommand{\mail}[1]{\href{mailto:#1}{\texttt{#1}}}

\allowdisplaybreaks

\begin{document}
\title{A PDE approach to Benamou--Brenier formula for the Schrödinger problem}
\author{M. Garatti\thanks{Université Paris-Saclay, CNRS, Laboratoire de mathématiques d’Orsay, ParMA, Inria Saclay, 91405, Orsay, France. e-mail: \mail{mattia.garatti@universite-paris-saclay.fr}}, L. Nenna\thanks{Universit\'e Paris-Saclay, CNRS, Laboratoire de math\'ematiques d'Orsay, ParMA, Inria Saclay, 91405, Orsay, France. 
e-mail: \mail{luca.nenna@universite-paris-saclay.fr}} \thanks{Institut Universitaire de France (IUF)}, S. Rota Nodari\thanks{Université Côte d'Azur, CNRS, Laboratoire J.A. Dieudonné,
06108 Nice, France. e-mail: \mail{simona.rotanodari@univ-cotedazur.fr}} 
\samethanks[3],
%\thanks{Institut Universitaire de France (IUF)},
L. Tamanini\thanks{Università Cattolica del Sacro Cuore, Dipartimento di Matematica e Fisica ``Niccol\`o Tartaglia'', I-25133, Brescia, Italy. e-mail: \mail{luca.tamanini@unicatt.it}}
}
\date{\today}

\newgeometry{left=3 cm, right = 3 cm, bottom=3 cm, top=2 cm}
\maketitle

\begin{abstract}
We studied the Benamou--Brenier formulation of the Schrödinger problem, focusing on a gap between theoretical results and applications, that often involve measures with unbounded support. While the existing proof in the literature relies on the compactness of the marginals' supports to ensure the necessary regularity of the Schrödinger potentials, we extend the validity of the Benamou--Brenier formula to the larger class of sub-Gaussian probability measures. Exploiting fine estimates on the Hessian of the potentials and the entropic interpolation, we provide an almost self-contained proof that establishes the existence of a velocity field with the appropriate polynomial growth that ensures the right integrability. This result justifies the use of the dynamic formulation in more general settings, such as Gaussian and mixture-of-Gaussians models, important also for the applications.
\end{abstract}

\vskip\baselineskip\noindent


%\tableofcontents

%\newpage

\section{Introduction}

The \emph{Schrödinger problem} originated in a series of papers by Erwin Schrödinger in 1931--1932 (see \cite{schrodinger1931umkehrung} and \cite{schrodinger1932theorie}), motivated by the following basic question: given the empirical distribution of a Brownian particle cloud at two distinct times, what is the most likely evolution connecting these two observations?  

Following \cite{follmer1997entropy}, \cite{leonard2013survey}, this naturally leads to the following variational problem: given two probability measures $\mu_0,\mu_1\in\mathcal P(\mathbb R^n)$ and a parameter $\varepsilon>0$, one aims to solve
\[
\boxed{\inf\bigl\{H(\gamma \mid \mathsf{R}_\varepsilon):\gamma \in\Gamma(\mu_0,\mu_1)\bigr\},}
\]
\restoregeometry
\noindent where the reference measure $\mathsf{R}_\varepsilon = \mathsf{r}_\varepsilon \mathcal{L}^n \otimes \mathcal{L}^n$ (see \eqref{eq:heatkernel} for the definition of $r_\varepsilon$) is the joint law at times $t=0$ and $t=\varepsilon$ of the (reversible) Wiener measure, $H(\cdot\mid\cdot)$ denotes the relative entropy, and $\Gamma(\mu_0,\mu_1)$ is the set of transport plans with marginals $\mu_0$ and $\mu_1$.

Although originally introduced in statistical mechanics, the Schr\"odinger problem has acquired renewed importance over the last two decades because of its deep connection with optimal transport. Indeed, as $\varepsilon\to0$, the Schr\"odinger problem converges to the quadratic optimal transport problem, both at the level of costs and minimizers; for this reason, it can be interpreted as an entropic regularization of the quadratic Monge--Kantorovich problem. 

This qualitative result was initially established by Mikami \cite{mikami2004monge} for the quadratic cost on $\mathbb{R}^n$, subsequently extended by Mikami–Thieullen \cite{mikami2008optimal} to more general cost functions and ultimately generalized by Léonard \cite{leonard2012schrodinger} to Polish spaces and general diffusion processes. From a more quantitative point of view, a huge literature on convergence rates is available (see \cite{carlier2023convergence, erbar2015large, adams2013large, conforti2021formula, chizat2020faster, pal2024difference, nenna2025convergence, nenna2025convergence_unb, eckstein2024convergence}). Understood the importance of the temperature parameter $\varepsilon$, it will be clear the decision of making it explicit throughout the paper.
A key feature of the Schr\"odinger problem is that, under mild assumptions, its unique minimizer $\gamma^\varepsilon$ admits the factorization
\[
\gamma^\varepsilon=(f^\varepsilon\otimes g^\varepsilon)\mathsf{R}_\varepsilon ,
\]
where the pair $(f^\varepsilon,g^\varepsilon)$ solves the so-called Schr\"odinger system (see, for details, \cite{ruschendorf1995convergence} and \eqref{SchoSystem} below). 
This decomposition plays a major role both in theory and in computation. On the one hand, it provides the natural analogue of Kantorovich potentials in the entropic framework; on the other hand, it allows to apply powerful iterative scaling procedures such as the Sinkhorn algorithm: see, for details, from the very first contributions \cite{sinkhorn1964relationship}, \cite{sinkhorn1967diagonal} and \cite{sinkhorn1967concerning} up to the more recents \cite{cuturi2013sinkhorn}, in which the entropic regularization and the Sinkhorn algorithm have really been introduced in the optimal transport theory as a computational tool, \cite{benamou2015iterative}, in which they show the applicability of the algorithm also to fluid-dinamics problems and \cite{greco2025hessian}, in which they establish new Hessian stability results and convergence rates for Sinkhorn potentials, using the semiconcavity of the potentials to refine the analysis of the algorithm's performance. In recent years, significant progress has been made in understanding the regularity, stability, and convexity properties of the corresponding Schr\"odinger potentials
\[
\varphi_t^\varepsilon=\varepsilon\log f_t^\varepsilon,
\qquad
\psi_t^\varepsilon=\varepsilon\log g_t^\varepsilon,
\]
as well as their relation to Hamilton--Jacobi--Bellman and Fokker--Planck equations; see for instance \cite{chiarini2023gradient}, in which they provide fundamental gradient estimates for Schrödinger potentials, which rigorously hold up the stability and convergence behavior of the Sinkhorn approximations toward Monge's problem; \cite{conforti2024weak}, in which he derives weak semiconvexity estimates for Schrödinger potentials, finding a direct link between these ones and the log-Sobolev inequality for Schrödinger bridges; and, again, \cite{greco2025hessian}.

The analogy with optimal transport becomes particularly striking at the dynamical level. In the quadratic transport setting, the celebrated Benamou--Brenier formula provides a dynamic reformulation of the Monge--Kantorovich problem:
\[
\boxed{
\begin{aligned}
\inf_{\gamma \in\Gamma(\mu_0,\mu_1)} & \int |x-y|^2\,\de\gamma(x,y)  \\
&= \inf\left\{ \int_0^1\int_{\mathbb R^n}|v_t|^2\,\eta_t\,\de\mathcal L^n\,\de t : \frac{\partial}{\partial t}\eta_t+\operatorname{div}(v_t\eta_t)=0, \eta_0=\mu_0, \eta_1=\mu_1 \right\}.
\end{aligned}
}
\]
First proved in \cite{benamou2000computational}, this formula is now one of the cornerstones of the theory of optimal transport; see, for example, \cite{villani2008optimal,santambrogio2015optimal,ambrosio2005gradient}. Its importance lies in the fact that it transforms a purely static minimization problem into a fluid-dynamical one, revealing the geometric structure of Wasserstein space and making possible a PDE-based approach to transport problems.
Because the Schr\"odinger problem is an entropic counterpart of optimal transport, it is natural to ask whether a similar dynamical reformulation holds in this setting. In the presence of diffusion encoded by the Wiener measure, the continuity equation should be replaced by the Fokker--Planck equation
\[
\frac{\partial}{\partial t}\eta_t-\frac{\varepsilon}{2}\Delta\eta_t+\operatorname{div}(v_t\eta_t)=0.
\]
One is then led to expect a Benamou--Brenier-type formula for the entropic cost, in which the relative entropy with respect to $\mathsf{R}_\varepsilon$ is represented as the minimum of a kinetic action under the Fokker--Planck constraint. Such a result was established by Gigli and Tamanini in \cite{gigli2020benamou} in the setting of $\mathrm{RCD}^*(K,N)$ spaces and under compactness assumptions on the marginals, namely bounded densities with bounded support. 
 We also refer to \cite{leonard2013survey,gentil2017analogy}, where the equivalence between the two formulations  with an only finite entropy assumption on the marginals is suggested via probabilistic arguments. However, the proof contains some flaws which can be fixed by using the same arguments as in \cite{cattiaux1995large}. 
Here we want to give an alternative, more explicit, and self-contained proof by only means of PDE tools.

% {\color{red}Notice that   in \cite{leonard2013survey} it is reported that the equivalence between the two formulations  with an only finite entropy assumption on the marginals by means of Girsanov theory. LN: reformulate this sentence}

% \todo[inline]{LT: allora, la situazione e' questa: io e Nicola dimostriamo BB per SP sotto ipotesi di compattezza. In $\R^n$ Gentil-L\'eonard-Ripani enunciano una formula di BB per SP, qui \cite[Theorem 5.1]{gentil2017analogy}, ma a me la loro dimostrazione sembra sbagliata. Piu' precisamente, sotto (5.7) quando prendono l'inf sulle misure di probabilita' P dovrebbero ottenere la disuguaglianza opposta. Confermate?

% Perche' faccio questa premessa? Perche' da un lato Gentil-L\'eonard-Ripani va citato, ma come citarlo? non so se dire che ``their proof contains a flaw'' o qualcosa del genere... perche' senza questa precisazione, sembra che loro facciano gia' BB per SP senza ipotesi di compattezza, quindi perche' farlo noi?}

However, many natural examples fall outside this compact framework. Gaussian measures, strongly log-concave measures, and various models used in applications have unbounded support. This is not merely a technical issue. In the Gaussian case, explicit computations are available and strongly suggest that the dynamic representation should remain valid beyond compactly supported marginals: see \cite{mallasto2022entropy}. More generally, probability distributions with sub-Gaussian tails arise naturally in stochastic analysis, Bayesian inference, diffusion models, and entropy-regularized transport between statistical distributions. Therefore, extending the Benamou--Brenier formulation to non-compact settings is both mathematically natural and relevant for applications.

The main difficulty is that the proof of \cite{gigli2020benamou} relies on regularity and integrability properties that are automatic under compact support but become delicate when the marginals are supported on the whole space. In particular, one needs to justify differentiation under the integral sign and repeated applications of Gauss--Green formulas for quantities involving the Schr\"odinger potentials and the entropic interpolation $\varrho^\varepsilon_t$. For this, it is essential to control expressions such as
\[
\psi_t^\varepsilon \rho_t^\varepsilon,\qquad
|\nabla\psi_t^\varepsilon|^2\rho_t^\varepsilon,\qquad
\psi_t^\varepsilon\Delta\rho_t^\varepsilon,
\]
and these controls are no longer immediate without compactness. Moreover, boundedness of the marginals is not inherited by the Schr\"odinger factors $f^\varepsilon$ and $g^\varepsilon$: as already visible in explicit Gaussian examples, one may have exponential growth of $g^\varepsilon$ at infinity even when the marginals themselves are smooth and rapidly decaying. This creates a genuine gap between the available abstract theory and the situations most relevant in practice.

The purpose of the present paper is to bridge this gap by identifying a sufficiently general class of marginals for which the Benamou--Brenier formula for the Schr\"odinger problem can still be proved rigorously. More precisely, we consider probability measures of the form
\[
\mu_0=e^{-V_0}\mathcal L^n,
\qquad
\mu_1=e^{-V_1}\mathcal L^n,
\]
where $V_0,V_1 \in C^2(\R^n)$ with uniform two-sided Hessian bounds. In particular, this framework includes Gaussian measures and, more generally, a broad class of strongly log-concave, sub-Gaussian distributions. The key feature of these assumptions is that they guarantee Gaussian-type decay at infinity while still allowing for unbounded support.
Our strategy is to combine recent regularity estimates for Schr\"odinger potentials with careful moment and integrability estimates along the entropic interpolation. On the regularity side, we exploit the semiconvexity and semiconcavity bounds, obtained in \cite{conforti2024weak} and \cite{chiarini2024semiconcavity}, which imply at most quadratic growth for the potentials and linear growth for their gradients. On the analytic side, we prove that the entropic interpolation has sufficiently strong Gaussian moments, together with suitable bounds on its derivatives, to justify all the integrations by parts required in the proof. In this way, we obtain an almost self-contained derivation of the Benamou--Brenier formula in a non-compact setting that still covers the principal examples of interest.
Practically speaking, the main contribution of this work can be summarized in finding sufficiently general assumptions on the marginals (the ones above) such that the following Theorem holds. Moreover, we also stress that we get a uniqueness result for the Benamou--Brenier formula under the Fokker--Planck constraint.

\begin{theorem}[Simplified version of Theorem \ref{main_theorem}]
Under suitable assumptions on $\mu_0,\mu_1 \in \mathcal{P}(\R^n)$, for every $\varepsilon >0$, it holds
\begin{align*}
\varepsilon &\underset{\gamma \in \Gamma(\mu_0,\mu_1)}{\min}   H(\gamma \mid \mathsf{R}_\varepsilon ) = \varepsilon H(\mu_0 \mid \mathcal{L}^n) \\
&+\min \left\lbrace \int_0^1 \int \frac{\abs{v_t}^2}{2}\eta_t \de\mathcal{L}^n \de t: \frac{\partial}{\partial t} \eta_t -\frac{\varepsilon}{2}\Delta \eta_t +\div(v_t \eta_t) =  0, \eta_0\mathcal{L}^n = \mu_0, \eta_1\mathcal{L}^n = \mu_1\right\rbrace,
\end{align*}
where the Fokker--Planck equation has to be understood in distributional sense.
\end{theorem}

\paragraph*{Outline of the contents} The papers is organized as follows: in Section \ref{Framework}, we collect a quick overview of the Schrödinger problem with the related theory and we establish the rigorous mathematical framework. In Section \ref{choice}, we describe how we arrived at the choice of sub-Gaussian assumptions for the marginals and we state our main result with a sketch of the proof and a direct Corollary. Then, in Section \ref{Main}, the first part is dedicated to prove all the technicalities and the second part to the rigorous proof of our main result, namely Theorem \ref{main_theorem}.

\section{Preliminaries and setting}\label{Framework}
In this Section we aim at establishing the mathematical framework as well as introducing the Schrödinger problem and some of its properties which will be helpful throughout the paper.

\paragraph{The Schrödinger problem} Consider, for a fixed $\varepsilon>0$, $\mathsf{R}_{\varepsilon} = \mathsf{r}_{\varepsilon} \mathcal{L}^n\otimes\mathcal{L}^n$, where
\begin{equation}\label{eq:heatkernel}
    \mathsf{r}_s(x,y)= \frac{1}{\sqrt{(2\pi s )^n}}e^{-\frac{\abs{x-y}^2}{2s}}
\end{equation}
is the heat kernel in $(x,y)$ at time $s$. Given $\mu_0= \varrho_0 \mathcal{L}^n, \mu_1=\varrho_1 \mathcal{L}^n \in \mathcal{P}_2(\R^n)$, namely the set of probability measures over $\R^n$ with finite second moment, the \emph{Schrödinger problem} reads as the minimization problem defined as
\begin{equation}\label{eq:SPstatic}\tag{$\mathrm{SP}_\varepsilon$}
\inf \left\lbrace H(\gamma \mid \mathsf{R}_{\varepsilon}): \gamma \in \Gamma(\mu_0,\mu_1)\right\rbrace,
\end{equation}
where 
\[
H(Q \mid P) = 
\begin{cases}
	\displaystyle \int \log\left( \frac{\de Q}{\de P}\right) \de Q & \textup{if } Q \ll P, \\[5pt]
	+\infty & \textup{otherwise}
\end{cases} 
\]
is the entropy functional and 
\[
\Gamma(\mu_0,\mu_1) = \left\lbrace \gamma \in \mathcal{P}(\R^n\times \R^n) : (p^1)_\# \gamma = \mu_0,  (p^2)_\# \gamma = \mu_1\right\rbrace, 
\]
where $p^1 : \R^n \times \R^n \to \R^n$, $p^1(x,y) := x$, and $p^2 :\R^n \times \R^n \to \R^n$, $p^2(x,y) := y$, are the canonical projections and where we use the standard notation $f_\# \gamma$ for any measurable map $f : \R^n \times \R^n \to \R^n$ to mean the push-forward measure defined as $f_\# \gamma(B) := \gamma(f^{-1}(B))$ for all Borel sets $B \subseteq \R^n$.

It is known (see \cite[Proposition $2.1$]{gigli2021second} for details, who build on the previous contributions \cite{leonard2001minimizers,borwein1994entropy,ruschendorf1993note}) that if $H(\mu_0 \otimes \mu_1 \mid \mathsf{R}_\varepsilon) < +\infty$, then there exists a unique solution $\gamma^\varepsilon$ to the above problem, also called \emph{Schr\"odinger plan}. Moreover, there exist two Borel functions $f^\varepsilon,g^\varepsilon: \R^n \to \left[ 0,+\infty\right[ $, a.e.\ uniquely determined up to the trivial transformation $(f,g) \mapsto (cf,\frac{g}{c})$, $c>0$, such that 
	\[ \gamma^\varepsilon = (f^\varepsilon \otimes g^\varepsilon)\mathsf{R}_{\varepsilon}, \]
%We will refer to the pair $(f^\varepsilon,g^\varepsilon)$ as $(f,g)$-decomposition.
where the couple $(f^\varepsilon,g^\varepsilon)$ is the unique solution (see  \cite[Section $2$]{ruschendorf1995convergence} or \cite[Theorem 2.1]{nutz2021introduction} ) to the so-called \emph{Schrödinger system}
\begin{equation}
\label{SchoSystem}
\begin{cases}
   \displaystyle f^\varepsilon(x) \int g^\varepsilon(y) \mathsf{r}_\varepsilon(x,y) \de y = \varrho_0(x) & \textup{ a.e. in } \R^n, \\[5pt]
   \displaystyle g^\varepsilon(y) \int f^\varepsilon(x) \mathsf{r}_\varepsilon(x,y) \de x = \varrho_1(y) & \textup{ a.e. in } \R^n. \\
\end{cases}\end{equation}
Notice that the above system simply force the $\gamma^\varepsilon$ to satisfy the marginal constraints, that is  $(f^\varepsilon \otimes g^\varepsilon)\mathsf{R}_\varepsilon \in \Gamma(\mu_0,\mu_1)$.

\paragraph{The Benamou--Brenier formula} In Optimal Transport, the Benamou--Brenier formula is a classical result that translates the static variational definition of Kantorovich into a fluid dynamical version that permits to recover the true meaning of this interpolation problem.

\begin{theorem}
For every $\mu_0,\mu_1 \in \mathcal{P}_2(\R^n)$, it holds
\begin{align*}\underset{\gamma \in \Gamma(\mu_0,\mu_1)}{\min}\; &\int \abs{x-y}^2 \de\gamma(x,y)\\ 
&=\min \left\lbrace \int_{0}^{1}\int \abs{v_t}^2 \eta_t \de\mathcal{L}^n \de t: \frac{\partial}{\partial t}\eta_t + \div(v_t \eta_t) = 0, \eta_0=\rho_0,\;\eta_1=\rho_1\right\rbrace,\end{align*}
where the continuity equation has to be understood in distributional sense.
\end{theorem}

\par\noindent As said, the previous Theorem has been firstly proved by \cite{benamou2000computational} and now it is one of the most important results of the theory (see  for instance \cite{ambrosio2021lectures,villani2008optimal,santambrogio2015optimal,ambrosio2005gradient} etc.).
Under some compactness assumptions on the marginals, the same results holds true also for the Schrödinger problem \cite{gigli2020benamou}, but the continuity equation constraint in the RHS is replaced by the Fokker--Planck equation.
 Let us then fix some notation to better comprehend this result. For every $x \in \R^n$, let us denote 
\begin{subequations}
\begin{equation}\label{eq:sol_heat_forward}
f^\varepsilon_t(x) = \begin{cases}
    f^\varepsilon(x) & \textup{if } t=0,\\
    \displaystyle \int f^\varepsilon(y)\mathsf{r}_{\varepsilon t}(x,y) \de y  & \textup{if } 0<t\le 1,
    \end{cases}
\end{equation}
\begin{equation}\label{eq:sol_heat_backward}
g^\varepsilon_t(x) = \begin{cases}
   \displaystyle\int g^\varepsilon(y)\mathsf{r}_{\varepsilon(1-t)}(x,y) \de y & \textup{if } 0\le t< 1, \\
   g^\varepsilon(x) & \textup{if } t=1,
\end{cases}
\end{equation}
\end{subequations}
namely the solutions of the heat equations
\begin{align*}
    &\begin{cases}
    \frac{\partial}{\partial t} f^\varepsilon_t - \frac{\varepsilon}{2} \Delta f^\varepsilon_t = 0 & \textup{in } \left] 0,1\right] \times \R^n, \\
    f^\varepsilon_0 = f^\varepsilon & \textup{on } \R^n,
\end{cases} & &\begin{cases}
    \frac{\partial}{\partial t} g^\varepsilon_t + \frac{\varepsilon}{2} \Delta g^\varepsilon_t = 0 & \textup{in } \left[ 0,1\right[ \times \R^n, \\
    g^\varepsilon_1 = g^\varepsilon & \textup{on } \R^n.
\end{cases}
\end{align*}
Furthermore, for all $t \in [0,1]$ let us define
\begin{equation}\label{eq:rho_mu}
    \varrho^\varepsilon_t := f_t^\varepsilon g_t^\varepsilon \qquad \textrm{and} \qquad \mu_t^\varepsilon := \varrho_t^\varepsilon \mathcal{L}^n \,,
\end{equation}
observing that $\varrho_i^\varepsilon = \varrho_i$ and $\mu_i^\varepsilon = \mu_i$, for $i=0,1$. Finally, let us denote for every $t \in \left]0,1\right[$ 
\begin{align}\label{eq:psi}
    \psi^\varepsilon_t &= \varepsilon \log g^\varepsilon_t
\end{align}
and $\psi^\varepsilon_1 (x) = \varepsilon \log g^\varepsilon(x)$ for every $x \in \supt (\mu_1)$, appropriately extended to $\R^n$ thanks to \eqref{SchoSystem}.

\begin{remark}
Since $f^\varepsilon_t, g^\varepsilon_t$ solve two heat equations, $\varrho^\varepsilon_t$ and $\psi^\varepsilon_t$ solve the following coupled system of Hamilton--Jacobi--Bellman/Fokker--Planck equations, that is
\[
\begin{cases}
\frac{\partial}{\partial t} \varrho^\varepsilon_t -\frac{\varepsilon}{2}\Delta \varrho^\varepsilon_t + \div (\varrho^\varepsilon_t \nabla \psi^\varepsilon_t) = 0 & \textup{in } \left] 0,1\right[ \times \R^n, \\
\frac{\partial}{\partial t} \psi^\varepsilon_t +\frac{\varepsilon}{2}\Delta \psi^\varepsilon_t +\frac{1}{2}\abs{\nabla \psi^\varepsilon_t}^2 =0 & \textup{in } \left[0,1\right[ \times \R^n, \\
\varrho^\varepsilon_0 = \varrho_0 & \textup{on } \R^n, \\
\varrho^\varepsilon_1 = \varrho_1 & \textup{on } \R^n, \\
\psi^\varepsilon_1 = \log g^\varepsilon & \textup{on }  \R^n.
\end{cases}\] \fr
\end{remark} 

Then, the main results in \cite{gigli2020benamou} can be summarized as follows.

\begin{theorem}
For every $\mu_0 = \varrho_0 \mathcal{L}^n, \mu_1 = \varrho_1 \mathcal{L}^n \in \mathcal{P}_2(\R^n)$ with bounded densities and supports, it holds
\begin{align*}
\varepsilon &\underset{\gamma \in \Gamma(\mu_0,\mu_1)}{\min} H(\gamma \mid \mathsf{R}_\varepsilon ) = \varepsilon H(\mu_0 \mid \mathcal{L}^n) \\
&+\min \left\lbrace \int_0^1 \int \frac{\abs{v_t}^2}{2}\eta_t \de\mathcal{L}^n \de t: \frac{\partial}{\partial t} \eta_t -\frac{\varepsilon}{2}\Delta \eta_t +\div(v_t \eta_t) =  0, \eta_0\mathcal{L}^n = \mu_0, \eta_1\mathcal{L}^n = \mu_1\right\rbrace,
\end{align*}
where the Fokker--Planck equation has to be understood in distributional sense.

In particular, on the right-hand side, a minimizer is $(\mu^\varepsilon_t,\nabla\psi^\varepsilon_t)$. 
\end{theorem}

\paragraph{Fokker--Planck pairs} In order to better understand in which sense the solution to Fokker--Planck equation has to be understood, we need to introduce some cut-off functions (note that we will use the same ones in the proof of the main result). 
Let $\lambda : [1,2] \to [0,1]$ be defined as
\[
\lambda (r) = \frac{e^{-\frac{1}{2-r}}}{e^{-\frac{1}{2-r}} +e^{-\frac{1}{r-1}}} \,, \qquad r \in\, ]1,2[
\]
and extended by continuity at $r=1,2$. As $\lambda$ is smooth, there exist $\Lambda \in \R$ such that for every $r \in [1,2]$
\[
\lambda'(r),\lambda''(r) \le \Lambda.
\]
Given $R>1$, consider the smooth function $\zeta_R : \R^n \to [0,1]$ such that
\begin{equation}\label{eq:cut_off}
\zeta_R(x) = \begin{cases}
    1 & \textup{ if } \abs{x} \le R,\\
    \displaystyle \lambda \left( \frac{\abs{x}}{R}\right) & \textup{ if } R < \abs{x} \le 2R,\\
    0 & \textup{ if } \abs{x} > 2R.
\end{cases}
\end{equation}
and observe that its gradient is
\[ \nabla\zeta_R(x) = \begin{cases}
    0 & \textup{ if } \abs{x} \le R,\\
    \displaystyle \frac{1}{R}\lambda' \left( \frac{\abs{x}}{R}\right) \frac{x}{\abs{x}} & \textup{ if } R < \abs{x} \le 2R,\\
    0 & \textup{ if } \abs{x} > 2R,
\end{cases} \]
and its Laplacian is
\[\Delta \zeta_R(x) = \begin{cases}
    0 & \textup{ if } \abs{x} \le R,\\
    \displaystyle \frac{1}{R^2}\lambda'' \left( \frac{\abs{x}}{R}\right)  +  \frac{n-1}{\abs{x}R}\lambda' \left( \frac{\abs{x}}{R}\right) & \textup{ if } R < \abs{x} \le 2R,\\
    0 & \textup{ if } \abs{x} > 2R.
\end{cases}\]
We need a convenient definition of distributional solution of the Fokker--Planck equation.

\begin{definition}\textbf{(distributional solution of Fokker--Planck equation)}
Let $v : [0,1] \times \R^n \to \R^n$ be a Borel map and $c>0$. We say that a curve $\nu \in C([0,1];\mathcal{P}(\R^n))$ is a \emph{distributional solution} of the \emph{Fokker--Planck equation} with velocity $v$, namely
\[ 
\frac{\partial}{\partial t} \nu-c\Delta \nu+\div(v \nu) =  0, 
\]
provided the following facts:
\begin{enumerate}[$(i)$]
    \item for every $t \in [0,1]$, $|v_t| \in L^1(\nu_t)$,
    \item the function $\left\lbrace t \mapsto \|v_t\|_{L^1(\nu_t)} \right\rbrace$ belongs to $L^1(0,1)$,
    \item for every $f \in C^1([0,1])\cap C^\infty_c(\R^n)$, it holds
\begin{equation}\label{eq:fokker-planck}
\int_{0}^{1} \int \left( \frac{\partial}{\partial t} f_t + c\Delta f_t + v_t\cdot \nabla f_t \right) \de\nu_t \de t= \int f_1 \de\nu_1 - \int f_0 \de\nu_0.
\end{equation}
\end{enumerate}
In particular, the pair $(\nu_t,v_t)$, where $\nu_t$ is a distributional solution of the forward Fokker--Planck equation defined by $v_t$, is called \emph{Fokker--Planck pair}.
\end{definition}

We observe that $(\mu^\varepsilon_t,\nabla\psi^\varepsilon_t)$ is a Fokker--Planck pair. 

\begin{remark}\label{weakeningtestfunctions}
It is worth noting that the class of test functions used in \eqref{eq:fokker-planck} can be extended from $C^1([0,1])\cap C^\infty_c(\R^n)$ to $C^1([0,1])\cap C_b^2(\R^n)$, that is if $(\nu_t,v_t)$ is a Fokker--Planck pair, then \eqref{eq:fokker-planck} holds true for all $f \in C^1([0,1])\cap C_b^2(\R^n)$. Passing from $C_c^\infty(\R^n)$ to $C_c^2(\R^n)$ is effortless, thanks to a mollification argument. Let us then explain how to pass from $C_c^2(\R^n)$ to $C_b^2(\R^n)$. Take $f \in C^1([0,1])\cap C_b^2(\R^n)$, consider, for $R>1$,
\[ f_R(t,x) = f(t,x)\zeta_R(x)\,,\]
where $\zeta_R$ is as in \eqref{eq:cut_off}, and use it as a test function. This gives
\begin{align*}
\int f_1\zeta_R \de\nu_1 -& \int f_0 \zeta_R \de\nu_0 - \int_0^1\int \zeta_R\left(\frac{\partial}{\partial t}f_t + c \Delta f_t + \nabla f_t \cdot v_t \right)\de\nu_t \de t \\
& = \int_0^1\int \left(cf_t\Delta\zeta_R + 2c \nabla \zeta_R \cdot \nabla f_t + f_tD \zeta_R \cdot v_t \right) \de\nu_t \de t. %=\\
%& = \int_0^1\int \left(\frac{c}{R^2}f_t\Delta\lambda + \frac{2c}{R}(\nabla\lambda)\left(\frac{x}{R}\right) \cdot \nabla f_t + \frac{f_t}{R} (\nabla\lambda)\left(\frac{x}{R}\right) \cdot v_t  \right) \de\nu_tdt.
\end{align*}
The left-hand side converges to a rearrangement of \eqref{eq:fokker-planck}, since $\zeta_R \to 1$ as $R \to +\infty$ and the right-hand side vanishes in the limit, since $\Delta\zeta_R, \nabla \zeta_R$ and $f$ are bounded by a constant and $\Delta\zeta_R, \nabla \zeta_R \to 0$ as $R \to +\infty$. \fr
\end{remark}

\section{The choice of marginals and the main result}\label{choice} 
As mentioned in the Introduction, one of the key points of the proof strategy in \cite{gigli2020benamou} is the boundedness of $g^\varepsilon$. This easily follows if both marginals have bounded and compactly supported densities. However, if we drop the compactness assumption on the supports, then $g^\varepsilon$ may fail to be bounded (even if so are both marginals' densities), as shown in the following example.

\begin{example}\label{counterexampleboundednesspsi}
Consider $\mu_0,\mu_1 \in \mathcal{P}_2(\R)$ defined as
\[
\mu_0 := \frac{1}{\sqrt{2\pi}}e^{-\frac{(x-1)^2}{2}} \mathcal{L}^1 \,, \qquad \mu_1 := \frac{1}{2\sqrt{\pi}}e^{-\frac{(y-2)^2}{4}} \mathcal{L}^1
\]
and $f,g : \R \to \left[0,+\infty\right[$ defined as
\[
f(x) := \frac{1}{\sqrt{2\pi}}e^{-\frac{x^2}{2}} \,, \qquad g(y) := e^{y-1} \,.
\]
We claim that the unique solution to \eqref{eq:SPstatic} with the above marginals and $\varepsilon=1$ is $\gamma := (f \otimes g)\mathsf{R}_{1}$. Indeed
\begin{align*}
    \frac{1}{\sqrt{2\pi}}e^{-\frac{x^2}{2}} \int_{-\infty}^{+\infty} e^{y-1}\frac{1}{\sqrt{2\pi}}e^{-\frac{(x-y)^2}{2}}\de y &= \frac{1}{2\pi} e^{-x^2 -1+\frac{(x+1)^2}{2}} \int_{-\infty}^{+\infty} e^{-\frac{(y-(x+1))^2}{2}}\de y \\
    &=\frac{1}{\sqrt{2\pi}}e^{-\frac{(x-1)^2}{2}}
\end{align*}
and 
\begin{align*}
    e^{y-1} \int_{-\infty}^{+\infty} \frac{1}{\sqrt{2\pi}}e^{-\frac{x^2}{2}}\frac{1}{\sqrt{2\pi}}e^{-\frac{(x-y)^2}{2}}\de x &= \frac{1}{2\pi} e^{-\frac{1}{4}y^2+y-1} \int_{-\infty}^{+\infty} e^{-(x-\frac{1}{2}y)^2}\de x = \frac{1}{2\sqrt{\pi}}e^{-\frac{(y-2)^2}{4}}.
\end{align*}
This shows that $\gamma \in \Gamma(\mu_0,\mu_1)$ and by \cite[Theorem 2.1]{nutz2021introduction}, this is sufficient to conclude that $\gamma$ is the Schr\"odinger plan between $\mu_0$ and $\mu_1$. \fr
\end{example}

%The previous example is a case of Gaussian marginals, so we already know the validity of a fluid dynamical, but not variational, representation formula (see \cite{mallasto2022entropy} for details). Notice that this is not in the framework of \cite{gigli2020benamou}, since Gaussians do not have compact support. The key point here is that boundedness of the marginals is not inherited by the decomposition of the solution: in fact, in this case, $g^\varepsilon$ is not bounded. In order to preserve the derivation under the integral sign argument, used in \cite{gigli2020benamou}, we need to find another way. 

%In addition, in order to choose a suitable, but sufficiently general, framework for the marginals, we have also to take into account the following case of Gaussian mixture marginals in which a fluid dynamical, but not variational, representation can be explicitly built. Again $\varepsilon = 1$ in order to simplify the computations.

Nonetheless, the fact that $g^\varepsilon$ may be unbounded is not an obstruction in itself. What is needed to replicate the proof of \cite{gigli2020benamou} is a suitable integrability property of the absolute value of the time derivative of $\psi_t^\varepsilon \varrho_t^\varepsilon$: therefore, even if $\psi_t^\varepsilon$ is unbounded, it is sufficient that $\varrho_t^\varepsilon$ compensates for its growth. This is indeed what happens in the next example, where we consider two marginals that are Gaussian mixtures (thus demonstrating that proper integrability also exists outside the Gaussian framework).

\begin{example}
Consider $\mu_0,\mu_1 \in \mathcal{P}_2(\R)$ defined as
\[
\mu_0 := \frac{1}{2\sqrt{2\pi}}\left(e^{-\frac{(x-1)^2}{2}} +e^{-\frac{(x+1)^2}{2}}\right) \mathcal{L}^1, \qquad \mu_1 :=  \frac{1}{4\sqrt{\pi}}\left(e^{-\frac{(y-2)^2}{4}} +e^{-\frac{(y+2)^2}{4}}\right) \mathcal{L}^1
\]
and $f,g : \R \to \left[0,+\infty\right[$ given by
\[
f(x) = \frac{1}{\sqrt{2\pi}}e^{-\frac{x^2}{2}} \,, \qquad g(y) = \frac{1}{2}\left( e^{y-1}+e^{-y-1}\right) \,.
\]
Then, arguing as in Example \ref{counterexampleboundednesspsi}, the unique solution to \eqref{eq:SPstatic} with the above marginals and $\varepsilon = 1$ is $\gamma = (f \otimes g)\mathsf{R}_{1}$. Given that, a direct computation provides
\[
f_t(x) = \frac{1}{\sqrt{2\pi(1+t)}}e^{-\frac{x^2}{2(1+t)}} \,, \qquad g_t(x) = \frac{1}{2}\left(e^{x-\frac{1+t}{2}}+ e^{-x-\frac{1+t}{2}}\right) \,,
\]
whence
\[\varrho_t(x) = \frac{1}{2\sqrt{2\pi(t+1)}}\left(e^{-\frac{(x-(1+t))^2}{2(1+t)}}+e^{-\frac{(x+(1+t))^2}{2(1+t)}}\right) \]
and
\[ \psi_t(x) = -\log 2 + \log \left(e^{x-\frac{1+t}{2}}+ e^{-x-\frac{1+t}{2}}\right).\]
Now, for every $t \in [0,1]$, the function $x \mapsto \psi_t(x)\varrho_t(x)$ is integrable and, for every $x \in \R$, the function $t \mapsto \psi_t(x)\varrho_t(x)$ is differentiable with
%, with 
% \begin{align*}
%     \frac{\partial}{\partial t} \left( \psi_t(x) \varrho_t(x) \right) &= \frac{\partial}{\partial t} \psi_t(x)\varrho_t(x) + \psi_t(x) \frac{\partial}{\partial t} \varrho_t(x) = \\
%     &=-\frac{1}{2}\varrho_t(x) + \psi_t(x) \left[ \frac{1}{\sqrt{2\pi}}\frac{\partial}{\partial t} \left( \frac{1}{\sqrt{1+t}}\left(e^{-\frac{(x-(1+t))^2}{2(1+t)}}+e^{-\frac{(x+(1+t))^2}{2(1+t)}}\right)\right)\right] = \\
%     &= -\frac{1}{2}\varrho_t(x) + \frac{\psi_t(x)}{(1+t)2\sqrt{2\pi(1+t)}}\left(e^{-\frac{(x-(1+t))^2}{2(1+t)}}+e^{-\frac{(x+(1+t))^2}{2(1+t)}}\right) + \\
%     &+\frac{\psi_t(x)}{2\sqrt{2\pi(1+t)}} \frac{\partial}{\partial t} \left(e^{-\frac{(x-(1+t))^2}{2(1+t)}}+e^{-\frac{(x+(1+t))^2}{2(1+t)}} \right) = \\
%     &= -\frac{1}{2}\varrho_t(x) + \frac{\psi_t(x)}{(1+t)}\varrho_t(x) + \\
%     &-\frac{\psi_t(x)}{2\sqrt{2\pi(1+t)}} \left[e^{-\frac{(x-(1+t))^2}{2(1+t)}} \frac{-4(x-(1+t))(1+t)-2(x-(1+t))^2}{4(1+t)^2} \right. + \\
%     &\left. + e^{-\frac{(x+(1+t))^2}{2(1+t)}}\frac{4(x+(1+t))(1+t)-2(x+(1+t))^2}{4(1+t)^2} \right] = \\
%     &= -\frac{1}{2}\varrho_t(x) + \frac{\psi_t(x)}{(1+t)}\varrho_t(x) - \frac{\psi_t(x)\left(-2x^2-6(1+t)^2\right)}{8(1+t)^2\sqrt{2\pi(1+t)}}e^{-\frac{(x-(1+t))^2}{2(1+t)}} +\\
%     &-\frac{\psi_t(x)\left(-2x^2+2(1+t)^2\right)}{8(1+t)^2\sqrt{2\pi(1+t)}}e^{-\frac{(x+(1+t))^2}{2(1+t)}}.
% \end{align*}
\[\abs*{\frac{\partial}{\partial t} \left( \psi_t(x) \varrho_t(x) \right)} \le 4 e^{-\frac{x^2}{100}} \,, \qquad \forall x \in \R.\]
We can thus differentiate under the integral sign and find, using the equations that $\psi_t$ and $\varrho_t$ solve and integration by parts,
\begin{align*}
    \frac{\de}{\de t} \int_{-\infty}^{+\infty} \psi_t \varrho_t  \de\mathcal{L}^1 &= \int_{-\infty}^{+\infty} \left( \rho_t\frac{\partial}{\partial t}\psi_t + \psi_t \frac{\partial}{\partial t} \varrho_t \right) \de\mathcal{L}^1 \\
    &= -\frac{1}{2} \int_{-\infty}^{+\infty} \Delta \psi_t \varrho_t \de\mathcal{L}^1 -\frac{1}{2}\int_{-\infty}^{+\infty} \abs*{\nabla \psi_t}^2\varrho_t \de\mathcal{L}^1 \\
    & \quad+\frac{1}{2}\int_{-\infty}^{+\infty} \psi_t \Delta\varrho_t \de\mathcal{L}^1 -\int_{-\infty}^{+\infty} \psi_t \operatorname{div} \left(\varrho_t \nabla\psi_t\right) \de\mathcal{L}^1 \\
    & = \frac{1}{2} \int_{-\infty}^{+\infty} \abs*{\nabla \psi_t}^2 \varrho_t \de\mathcal{L}^1 \,.
\end{align*}
If we further integrate both sides w.r.t.\ $t \in [0,1]$, we find
\[\int \psi_1 \de\mu_1 - \int \psi_0 \de\mu_0 = \int_0^1\int \frac{\abs{\nabla \psi_t}^2}{2}\varrho_t \de\mathcal{L}^1 \de t. \]
Finally, noticing that $\psi_0 = \log \varrho_0 - \log f$ and $\psi_1 = \log g$, and using also the marginal constraints on $\gamma$, we arrive at 
\[ H(\gamma \mid \mathsf{R}_1) =  H(\mu_0 \mid \mathcal{L}^1)+\int_0^1 \int \frac{\abs{\nabla\psi_t}^2}{2}\de\mu_t \de t.\] \fr
\end{example}

With this example in mind we are now ready to get the right assumptions for our framework: the idea is to preserve the property of having very little mass at infinity instead of the strongest assumption of the compact support, in order to obtain the necessary integrability.

\begin{assumptions}
\label{assumptions}
Suppose $\mu_0,\mu_1 \in \mathcal{P}_2(\R^n)$ have the form 
\begin{align*}
    \mu_0 &= e^{-V_0} \mathcal{L}^n, & \mu_1 &= e^{-V_1} \mathcal{L}^n,
\end{align*}
with $V_0, V_1 \in C^2(\R^n)$ for which there exist $C>0$ such that, for every $x,\xi \in \R^n$,
\begin{align*}
    C^{-1}\abs*{\xi}^2 \le \xi \cdot D^2 V_0(x) \xi  \le C \abs*{\xi}^2, \\ 
    C^{-1}\abs*{\xi}^2 \le \xi \cdot D^2 V_1(x) \xi \le C \abs*{\xi}^2.
\end{align*} 
\end{assumptions}

\begin{remark}
It is possible to include more general cases in our framework if we consider $\mu_0,\mu_1 \in \mathcal{P}_2(\R^n)$ having the form 
\begin{align*}
    \mu_0 = \begin{cases}
        e^{-V_0} \mathcal{L}^n & \textup{if } V_0 \textup{ is defined},\\
        0 & \textup{otherwise},
    \end{cases}  & & \mu_1 &= \begin{cases}
        e^{-V_1} \mathcal{L}^n & \textup{if }V_1 \textup{ is defined},\\
        0 & \textup{otherwise},
    \end{cases}
\end{align*}
with $V_0 \in C^2(\textup{dom}(V_0)), V_1 \in C^2(\textup{dom}(V_1))$, $\textup{dom}(V_0), \textup{dom}(V_1) \subseteq \R^n$ non-empty, possibly unbounded and with a finite number of convex connected components, for which there exist $C>0$ such that
\begin{align*}
    C^{-1}\abs*{\xi}^2 \le \xi \cdot D^2 V_0(x) \xi  \le C \abs*{\xi}^2, \qquad \textup{for every } x \in \textup{dom}(V_0), \xi \in \R^n, \\ 
    C^{-1}\abs*{\xi}^2 \le \xi \cdot D^2 V_1(x) \xi \le C \abs*{\xi}^2,\qquad \textup{for every } x \in \textup{dom}(V_1), \xi \in \R^n.
\end{align*}
With this slightly more general framework, it is possible to recover the case of compactly supported marginals, namely the result in \cite{gigli2020benamou}, and also manage the case of marginals with non-connected supports. The difference in the proof is minor: while the underlying strategy remains unchanged, one should be more careful when applying Taylor expansions and on writing the domain of $\psi^\varepsilon_1$. It is merely a matter of expanding on technical details that are not essential to the true core of the proof. Therefore, for the sake of presentation we prefer to adopt Assumptions \ref{assumptions} throughout the whole manuscript. \fr
\end{remark}

%First of all, in an attempt of completeness, we can ensure, at least, well definiteness of the Schrödinger problem and existence, uniqueness, and the decomposition of the solution: since the reference measure $\mathsf{R}_\varepsilon$ is unbounded, the proof is basically the checking of the well-definiteness of the entropy functional and, according to \cite[Proposition $2.1$]{gigli2021second} and \cite{leonard2013survey}, of the validity of $H(\mu_0 \otimes \mu_1 \mid \mathsf{R}_\varepsilon) < +\infty$.

First of all, let us verify that under these assumptions existence, uniqueness, and decomposition of the solution to \eqref{eq:SPstatic} hold true.  

\begin{proposition}\label{Setting}
Under Assumptions \ref{assumptions}, for every $\varepsilon >0$ there exists a unique solution $\gamma^\varepsilon$ to \eqref{eq:SPstatic}. 
%the entropy functional is well defined\todo{e' ben definito dove?} and $H(\mu_0 \otimes \mu_1 \mid \mathsf{R}_\varepsilon) < +\infty$. Moreover, there exists a unique $\gamma^\varepsilon \in \Gamma(\mu_0,\mu_1)$ such that 
%\[ H(\gamma^\varepsilon \mid \mathsf{R}_{\varepsilon}) = \min \left\lbrace H(\gamma \mid \mathsf{R}_{\varepsilon}): \gamma \in \Gamma(\mu_0,\mu_1)\right\rbrace,\]
%where 
%\[\Gamma(\mu_0,\mu_1) = \left\lbrace \gamma \in \mathcal{P}(\R^n\times \R^n) : (p^1)_\# \gamma = \mu_0,  (p^2)_\# \gamma = \mu_1\right\rbrace. \]
In particular, there exist two Borel functions $f^\varepsilon,g^\varepsilon: \R^n \to \left[ 0,+\infty\right[ $, unique \emph{a.e.} in $\R^n $ up to a rescaling $(f,g) \mapsto (cf,\frac{g}{c})$ for some $c>0$, such that 
	\[ \gamma^\varepsilon = (f^\varepsilon \otimes g^\varepsilon)\mathsf{R}_{\varepsilon}. \]
\end{proposition}

\begin{proof}
According to \cite[Proposition $2.1$]{gigli2021second}, if we exhibit a Borel function $B: \R^n\times \R^n \to \left[ 0,+\infty\right[ $ such that
\begin{equation}\label{eq:technical}
\int e^{-B(x)-B(y)}\de\mathsf{R}_\varepsilon(x,y) < +\infty, \qquad \int B \de\mu_0 <+\infty, \qquad  \int B \de\mu_1 <+\infty
\end{equation} 
and show that $H(\mu_0 \otimes \mu_1 \mid \mathsf{R}_\varepsilon) < +\infty$, then all the claimed properties hold true.

As concerns \eqref{eq:technical}, since the marginals have finite second moment, the choice $B(x) = \abs{x}^2$ satisfies the requirements. As for the feasibility of the marginal constraint, note that
\begin{align*}
H(\mu_0 \otimes \mu_1 \mid \mathsf{R}_\varepsilon) &= H(\mu_0 \otimes \mu_1 \mid \mathcal{L}^n \otimes \mathcal{L}^n) + \int \log \left( \sqrt{2\pi\varepsilon} e^{\frac{\abs{x-y}^2}{2\varepsilon}} \right) \de (\mu_0\otimes \mu_1 ) \\
&=H(\mu_0 \mid \mathcal{L}^n) + H(\mu_1 \mid \mathcal{L}^n) + \int \log \left( \sqrt{2\pi\varepsilon} e^{\frac{\abs{x-y}^2}{2\varepsilon}} \right) \de (\mu_0\otimes \mu_1 ) < +\infty.
\end{align*}
Thus \cite[Proposition $2.1$]{gigli2021second} grants existence and uniqueness of the solution, as well as the existence of the decomposition.
\end{proof}

We are now ready to state our main result. A complete proof is postponed after Proposition \ref{fluidnotvar} because we need to deal at first with some technical Lemmas.

\begin{theorem}\emph{\textbf{(Benamou--Brenier formula for the Schrödinger problem)}}\label{main_theorem}
Under Assumptions \ref{assumptions}, for every $\varepsilon >0$ it holds
\begin{align*}
\varepsilon &\underset{\gamma \in \Gamma(\mu_0,\mu_1)}{\min} H(\gamma \mid \mathsf{R}_\varepsilon ) = \varepsilon H(\mu_0 \mid \mathcal{L}^n) \\
&+\min \left\lbrace \int_0^1 \int \frac{\abs{v_t}^2}{2}\eta_t \de\mathcal{L}^n \de t: (\eta_t\mathcal{L}^n,v_t) \textup{  Fokker--Planck pair, } \eta_0 = \varrho_0, \eta_1 = \varrho_1\right\rbrace.
\end{align*}
Moreover, on the right-hand side, the (unique) minimizer is $(\mu^\varepsilon_t,\nabla\psi^\varepsilon_t)$. 
\end{theorem}
\begin{proof}\textit{(Sketch)} The proof of the equality can be summarized in two main parts: the first is to prove 
\[  
\varepsilon H(\gamma^\varepsilon \mid \mathsf{R}_\varepsilon ) = \varepsilon H(\mu_0 \mid \mathcal{L}^n) + \int_0^1\int \frac{\abs{\nabla\psi_t^\varepsilon}^2}{2}\de\mu^\varepsilon_t \de t
\]
that permits us to obtain the $\ge$ inequality. The second part is to argue by duality to obtain the $\le$ inequality. 

We postpone the proof of the uniqueness of the minimizer.
\end{proof}

As a Corollary, we can also obtain a dual formulation for the Schrödinger problem.
\begin{corollary}
Under Assumptions \ref{assumptions}, for every $\varepsilon >0$ it holds
\begin{align*}
\varepsilon &\underset{\gamma \in \Gamma(\mu_0,\mu_1)}{\min} H(\gamma \mid \mathsf{R}_\varepsilon ) = \varepsilon H(\mu_0 \mid \mathcal{L}^n) \\
&+\max \left\lbrace \int f_1 \de\mu_1 - \int f_0 \de \mu_0: f \in \mathfrak{H}, \frac{\partial}{\partial t} f + \frac{\varepsilon}{2}\Delta f + \frac{1}{2}\abs{\nabla f}^2 \le 0\right\rbrace \,,
\end{align*}
where 
\[\mathfrak{H} = \left\lbrace f \in C^1([0,1])\cap C^2(\R^n):
    m^\varepsilon_t \abs{\xi}^2 \le \xi \cdot D^2f(t,x) \xi \le M^\varepsilon_t\abs{\xi}^2 \right\rbrace,\]
where $m^\varepsilon_t,M^\varepsilon_t \in \R$ and depend on the marginals as in Lemma \ref{psi_t} (b).
Moreover, on the right-hand side, a maximizer is given by $\psi^\varepsilon_t$. 
\end{corollary}
\begin{proof}
It is a direct consequence of Lemma \ref{density}, Proposition \ref{fluidnotvar}, and the previous Theorem.
\end{proof}

\section{Proofs}\label{Main}
\subsection{Technical results}
In the following, we collect almost all the technical parts of the document, some of which we believe are interesting in their own right.

The first technical result is on the growth of $\psi^\varepsilon_1 = \varepsilon \log g^\varepsilon$. We are able to prove, thanks to the results of \cite{chiarini2024semiconcavity},  that it inherits Hessian bounds from the ones of the marginals.
\begin{lemma}\label{psi_1}
Let $\mu_0$ and $\mu_1$ satisfy Assumptions \ref{assumptions} and $\psi_1^\varepsilon=\varepsilon\log g^\varepsilon$, with $g^\varepsilon$ a solution to \eqref{SchoSystem}. Then for every $\varepsilon >0$, the following facts hold true:
\begin{enumerate}[$(a)$]
    \item $\psi^\varepsilon_1 \in C^2(\R^n)$,
    \item for every $x,\xi \in \R^n$,
    \[ m^\varepsilon\abs{\xi}^2 \le \xi \cdot D^2\psi^\varepsilon_1(x) \xi \le M^\varepsilon\abs{\xi}^2,\]
    where 
    \begin{align*}
        m^\varepsilon &=  -1 -\frac{\varepsilon}{2C} + \frac{1}{C}\sqrt{\frac{\varepsilon^2}{4} + 1}, & M^\varepsilon &= -1 -\frac{\varepsilon}{2}C + C\sqrt{\frac{\varepsilon^2}{4} + 1},
    \end{align*}
    and $C$ is as in Assumptions \ref{assumptions},
    \item for every $x \in \R^n$,
    \[ \abs{\psi^\varepsilon_1(x)} \le c^\varepsilon + b^\varepsilon \abs{x} + \frac{a^\varepsilon}{2} \abs{x}^2,\]
    where $c^\varepsilon = \abs{\psi^\varepsilon_1(0)}$, $b^\varepsilon = \abs{\nabla\psi^\varepsilon_1(0)}$, and $a^\varepsilon = \max \left\lbrace \abs{m^\varepsilon}, \abs{M^\varepsilon} \right\rbrace$.
\end{enumerate}
\end{lemma}

\begin{proof}\hfill
\par\noindent 
$(a)$ Applying $\log$ to both sides of the second equation of the Schrödinger system and recalling \eqref{eq:sol_heat_forward}, we can write $\psi^\varepsilon_1 = \varepsilon\log g^\varepsilon = \varepsilon\log \varrho_1 - \varepsilon \log f^\varepsilon_1$, whence the desired regularity.

\par\noindent 
$(b)$ Since, by $(a)$, $\psi^\varepsilon_1 \in C^2(\R^n)$, arguing as in \cite[Corollary $2.3$]{chiarini2024semiconcavity}, we have, for every $x,\xi \in \R^n$,
\[m^\varepsilon \abs{\xi}^2 \le \xi \cdot D^2\psi^\varepsilon_1(x) \xi \le M^\varepsilon\abs{\xi}^2.\]
\par\noindent
$(c)$ Using Taylor's formula with integral remainder, for every $x \in \R^n$,
\[\psi^\varepsilon_1(x) = \psi^\varepsilon_1(0) + \nabla\psi^\varepsilon_1(0)\cdot x + \int_0^1 (1-t) x \cdot D^2 \psi^\varepsilon_1(tx) x \de t,\]
and the inequality from $(b)$
% But
% \begin{align*}
%     \left( -1 -\frac{\varepsilon}{2}C_3 + \sqrt{\frac{\varepsilon^2}{4}C_3^2 + \frac{C_3}{C_2}}\right) &\int_0^1 (1-t) \de t \abs{x}^2 \le \int_0^1 (1-t) x \cdot D^2 \psi^\varepsilon_1(tx) x \de t \le \\ 
%     &\le \left( -1 -\frac{\varepsilon}{2}C_4 + \sqrt{\frac{\varepsilon^2}{4}C_4^2 + \frac{C_4}{C_1}}\right) \int_0^1 (1-t) \de t \abs{x}^2,
% \end{align*}
% that is 
% \begin{align*}
%     \frac{1}{2}\left( -1 -\frac{\varepsilon}{2}C_3 + \sqrt{\frac{\varepsilon^2}{4}C_3^2 + \frac{C_3}{C_2}}\right) \abs{x}^2 &\le \int_0^1 (1-t) x \cdot D^2 \psi^\varepsilon_1(tx) x \de t \le \\ 
%     &\le  \frac{1}{2}\left( -1 -\frac{\varepsilon}{2}C_4 + \sqrt{\frac{\varepsilon^2}{4}C_4^2 + \frac{C_4}{C_1}}\right) \abs{x}^2,
% \end{align*}
we get
\begin{align*}
  \psi^\varepsilon_1(0) + \nabla\psi^\varepsilon_1(0)\cdot x + \frac{1}{2}m^\varepsilon \abs{x}^2 \le \psi^\varepsilon_1(x) \le \psi^\varepsilon_1(0) + \nabla\psi^\varepsilon_1(0)\cdot x + \frac{1}{2}M^\varepsilon \abs{x}^2,
\end{align*}
then, by Cauchy--Schwarz inequality and some elementary estimates with the absolute value and the maximum between real numbers,
\[  -c^\varepsilon - b^\varepsilon \abs{x} - a^\varepsilon \abs{x}^2 \le \psi^\varepsilon_1(x) \le c^\varepsilon + b^\varepsilon \abs{x} + a^\varepsilon \abs{x}^2\]
and the result follows.
\end{proof}

In this second Lemma, we apply the important result contained in \cite{conforti2024weak} to propagate the Hessian bounds on $\psi^\varepsilon_1$ to every time in $[0,1]$.

\begin{lemma}\label{psi_t}
Let $\mu_0$ and $\mu_1$ satisfy Assumptions \ref{assumptions} and $\psi_t^\varepsilon=\varepsilon\log(g_t^\varepsilon)$, with $g_t^\varepsilon$ a solution to the backward heat equation with datum $g^\varepsilon$. Then, for every $\varepsilon >0$, the following facts hold true:
\begin{enumerate}[$(a)$]
    \item $\psi^\varepsilon_t \in C^1([0,1])\cap C^2(\R^n)$,
    \item for every $t \in [0,1]$, for every $x,\xi \in \R^n$,
    \[ \frac{m^\varepsilon}{1+(1-t)m^\varepsilon}\abs{\xi}^2 \le \xi \cdot D^2\psi^\varepsilon_t(x) \xi \le \frac{M^\varepsilon}{1+(1-t)M^\varepsilon}\abs{\xi}^2,\]
    \item for every $t \in [0,1]$, for every $x \in \R^n$,
    \[ \abs{\nabla\psi^\varepsilon_t(x)} \le b^\varepsilon_t + a^\varepsilon_t\abs{x},\]
    where $b^\varepsilon_t = \abs{\nabla\psi^\varepsilon_t(0)}$ and 
    \[ a^\varepsilon_t =\max \left\lbrace \abs*{\frac{m^\varepsilon}{1+(1-t)m^\varepsilon}}, \abs*{ \frac{M^\varepsilon}{1+(1-t)M^\varepsilon}} \right\rbrace, \]
    \item for every $t \in [0,1]$, for every $x \in \R^n$,
    \[ \abs{\psi^\varepsilon_t(x)} \le c^\varepsilon_t + b^\varepsilon_t \abs{x} + \frac{a^\varepsilon_t}{2} \abs{x}^2,\]
    where $c^\varepsilon_t = \abs{\psi^\varepsilon_t(0)}$.
\end{enumerate}
\end{lemma}
\begin{proof}\hfill
\par\noindent 
$(a)$ It comes from the very definition of $\psi^\varepsilon_t$. Indeed, the latter is obtained from a convolution against the heat kernel.
\par\noindent
$(b)$ It is a direct application of \cite{conforti2024weak}.
\par\noindent
$(c)$ Note that, for all $x \in \R^n$, it holds
\begin{align*}
|\nabla\psi^\varepsilon_t (x)| = \left|\nabla \psi^\varepsilon_t (0) + \int_0^1 D^2 \psi^\varepsilon_t (sx)x \,\de s\right| \leq \abs{\nabla \psi^\varepsilon_t (0)} + \int_0^1 \abs*{D^2 \psi^\varepsilon_t (sx)} \abs{x} \,\de s \,,
\end{align*}
so that the conclusion follows from (b).
\par\noindent
$(c)$ It is the same argument as in Lemma \ref{psi_1} (c).
\end{proof}

In the following, we shall use the following result (for more details see for instance \cite[Lemma 1]{brezis1981nonlinear} and \cite{kato1972schrodinger}).
\begin{lemma}\emph{\textbf{(parabolic Kato's inequality)}}
Let $U$ be an open subset of $\R \times \R^n$ and $u \in L^1_{loc}(U)$ a weak solution of
\[\frac{\partial}{\partial t} u -\Delta u = f,\]
with $f \in L^1_{loc}(U)$. Then, weakly,
\[\frac{\partial}{\partial t} \abs{u} -\Delta \abs{u} \le \textup{sgn} (u) f ,\]
where
\[\textup{sgn} (u) = \begin{cases}
    1 & \textup{ if } u> 0, \\
    0  & \textup{ if } u= 0, \\
    -1 & \textup{ if } u< 0. \\
\end{cases}\]
\end{lemma}

In the following technical result, which we believe to be interesting also in itself, we establish that, under some bounds on the velocity fields, a solution to the Fokker--Planck equation and its spatial derivatives have finite Gaussian and $k$-th moments, respectively.

\begin{proposition}\label{varrho_t}
Let $\mu_0$ satisfying Assumptions \ref{assumptions} and $(\rho_t^\varepsilon\mathcal{L}^n, \nabla\psi^\varepsilon_t)$ be a Fokker--Planck pair with initial datum $\mu_0$ such that $\psi^\varepsilon_t \in C^1([0,1])\cap C^2(\R^n)$ and, for every $t \in [0,1]$, for every $x,\xi \in \R^n$,
    \[ m^\varepsilon_t\abs{\xi}^2 \le \xi \cdot D^2\psi^\varepsilon_t(x) \xi \le M^\varepsilon_t\abs{\xi}^2,\]
where $m^\varepsilon_t, M^\varepsilon_t$ are continuous functions on $[0,1]$. Then, for every $\varepsilon >0$, the following facts hold true:
\begin{enumerate}[$(a)$]
    \item $\varrho^\varepsilon_t \in C^1([0,1])\cap C^2(\R^n)$,
    \item there exists $ W : [0,1] \to \R $ such that $0 < W(1) \le W < C_2$ and, for every $R >1$,
\[ \underset{t \in [0,1]}{\sup}\; \int_{\R^n \setminus \B(0,R)} \varrho^\varepsilon_t \de\mathcal{L}^n \le I e^{-W(1)R^2},\]
where 
\[I = \underset{t \in [0,1]}{\sup}\; \int e^{W(t) \abs{x}^2} \varrho^\varepsilon_t(x) \de x < +\infty,\]
\item for every $k \in \N$,
\[\underset{t \in [0,1]}{\sup}\; \int \abs{x}^k \varrho^\varepsilon_t(x) \de x  < +\infty.\]
Moreover,
\begin{align*}
& \underset{t \in [0,1]}{\sup}\; \int \left(1+\abs{x}^2\right)\varrho^\varepsilon_t(x) \de x  < +\infty, \\
& \underset{t \in [0,1]}{\sup}\; \int \left(1+\abs{x}^2\right)\abs{\nabla\varrho^\varepsilon_t(x)} \de x  < +\infty \\
& \underset{t \in [0,1]}{\sup}\; \int \left(1+\abs{x}^2\right)\abs{\Delta \varrho^\varepsilon_t(x)} \de x  < +\infty.
\end{align*}
\end{enumerate}
\end{proposition}
\begin{proof}
We firstly notice that, arguing in the same way as in Lemma \ref{psi_t}, for every $t \in [0,1]$, for every $x \in \R^n$,
    \[ \abs{\nabla\psi^\varepsilon_t(x)} \le b^\varepsilon_t + a^\varepsilon_t\abs{x},\]
where $a^\varepsilon_t, b^\varepsilon_t$ are continuous functions on $[0,1]$. So, in particular, for every $t \in [0,1]$, for every $x \in \R^n$,
    \[ \abs{\psi^\varepsilon_t(x)} \le c^\varepsilon_t + b^\varepsilon_t \abs{x} + \frac{a^\varepsilon_t}{2} \abs{x}^2,\]
    where $c^\varepsilon_t$ is another continuous function on $[0,1]$.
\par\noindent 
$(a)$ It follows by definition of $\varrho^\varepsilon_t$.
\par\noindent
$(b)$ First of all, given $k \in \left]0, C_2\right[$, take 
\[W(t) = k \dfrac{ A e^{-At} }{A + 2k\varepsilon \left( 1 - e^{-At}\right)},\]
where 
\[A = \frac{1}{2} \underset{t \in [0,1]}{\max}\; b^\varepsilon_t + 2\underset{t \in [0,1]}{\max}\; a^\varepsilon_t,\]
and $a^\varepsilon_t$ and $b^\varepsilon_t$ are given by Lemmas \ref{psi_1} and \ref{psi_t}.
One can show that $W(t)$ solves the following Cauchy-Lipschitz problem
\[\begin{cases}
    W' + 2 \varepsilon W^2 +AW =0,\\
    W(0) = k
\end{cases}\]
and $ 0 < W(1) \le W < C_2$. Now,  given $N >0$, introducing, for ease of notation, 
\[\Phi_R^N(x)=\zeta_R(x) \min\left\lbrace e^{W(t)\abs{x}^2}, N\right\rbrace,\]
using the derivation under the integral sign, the equations solved by $\varrho^\varepsilon_t$, and integrating by parts we obtain
\begin{align*}
    & \frac{\de}{\de t} \int \Phi_R^N(x) \varrho^\varepsilon_t(x) \de x \le \int W'(t) \abs{x}^2 \zeta_R(x) e^{W(t)\abs{x}^2}\varrho^\varepsilon_t(x)\de x \\ 
    & \quad+ \int  \Phi_R^N(x)\left( \frac{\varepsilon}{2}\Delta \varrho^\varepsilon_t - \div (\varrho^\varepsilon_t \nabla \psi^\varepsilon_t)\right)\de x \\
    & \le \int W'(t) \abs{x}^2 \zeta_R(x) e^{W(t)\abs{x}^2}\varrho^\varepsilon_t(x)\de x + \frac{\varepsilon}{2}\int  \Delta \left( \Phi_R^N(x)\right) \varrho^\varepsilon_t \de x \\ 
    & \quad + \int \nabla\left(\Phi_R^N(x) \right) \cdot \nabla \psi^\varepsilon_t \varrho^\varepsilon_t \de x.
\end{align*}
Using Leibniz's rule for gradient and Laplacian, the property of exponential, Cauchy--Schwarz inequality, since
\[ \int_{\left\lbrace e^{W(t)\abs{x}^2} = N\right\rbrace} \Delta \min\left\lbrace e^{W(t)\abs{x}^2}, N\right\rbrace \varrho^\varepsilon_t \de x =-\int 2N \sqrt{\log N} \varrho^\varepsilon_td\mathcal{H}^{n-1} \le 0 \]
% \begin{align*}
%     \int_{\left\lbrace e^{W(t)\abs{x}^2} = N\right\rbrace} &\Delta \min\left\lbrace e^{W(t)\abs{x}^2}, N\right\rbrace \varrho^\varepsilon_t \de x = \\
%     &=\int \nabla  \min\left\lbrace e^{W(t)\abs{x}^2}, N\right\rbrace \cdot \nu(x) \mid_{\abs{x}^2=\log N} \varrho^\varepsilon_t(x) \de\mathcal{H}^{n-1}(x) = \\
% &=  -\int 2N \sqrt{\log N} \varrho^\varepsilon_td\mathcal{H}^{n-1} \le 0
% \end{align*}
we get
\begin{align*}
    &\frac{\varepsilon}{2}\int  \Delta  \left( \Phi_R^N(x)\right)  \varrho^\varepsilon_t \de x \\
    &\le \varepsilon \int (n+2W(t)\abs{x}^2) W(t) \Phi_R^N(x) \varrho^\varepsilon_t \de x + \frac{\varepsilon}{2}\frac{\Lambda}{R^2}n \int_{\B(0,2R)\setminus\B(0,R)} \min\left\lbrace e^{W(t)\abs{x}^2}, N\right\rbrace \varrho^\varepsilon_t \de x \\
    & \quad + 4\varepsilon \int_{\B(0,2R)\setminus\B(0,R)} \Lambda W(t) \min\left\lbrace e^{W(t)\abs{x}^2}, N\right\rbrace \varrho^\varepsilon_t \de x
\end{align*}
% \begin{align*}
%      &\frac{\varepsilon}{2}\int  \Delta  \left( \Phi_R^N(x)\right)  \varrho^\varepsilon_t \de x  = \frac{\varepsilon}{2} \int \zeta_R(x) \Delta \min\left\lbrace e^{W(t)\abs{x}^2}, N\right\rbrace \varrho^\varepsilon_t \de x + \\ 
%     &+ \varepsilon \int \nabla\zeta_R(x) \cdot \nabla \min\left\lbrace e^{W(t)\abs{x}^2}, N\right\rbrace \varrho^\varepsilon_t \de x  + \frac{\varepsilon}{2} \int \Delta \zeta_R(x)  \min\left\lbrace e^{W(t)\abs{x}^2}, N\right\rbrace \varrho^\varepsilon_t \de x \le \\
%     & \le \varepsilon \int n W(t) \Phi_R^N(x) \varrho^\varepsilon_t \de x + 2\varepsilon \int \zeta_R(x) W(t)^2 \abs{x}^2 \min\left\lbrace e^{W(t)\abs{x}^2}, N\right\rbrace \varrho^\varepsilon_t \de x + \\
%     &+ 4\varepsilon \int_{\B(0,2R)\setminus\B(0,R)} \Lambda W(t) \min\left\lbrace e^{W(t)\abs{x}^2}, N\right\rbrace \varrho^\varepsilon_t \de x +\\ &\frac{\varepsilon}{2}\frac{\Lambda}{R^2}n \int_{\B(0,2R)\setminus\B(0,R)} \min\left\lbrace e^{W(t)\abs{x}^2}, N\right\rbrace \varrho^\varepsilon_t \de x
% \end{align*}
and 
\newpage
\begin{align*}
    \int \nabla&\left(\Phi_R^N(x) \right) \cdot \nabla \psi^\varepsilon_t \varrho^\varepsilon_t \de x \\
    &= \int \left(\nabla\zeta_R(x)  \cdot \nabla \psi^\varepsilon_t \min\left\lbrace e^{W(t)\abs{x}^2}, N\right\rbrace + \zeta_R(x) \nabla \min\left\lbrace e^{W(t)\abs{x}^2}, N\right\rbrace \cdot \nabla \psi^\varepsilon_t\right) \varrho^\varepsilon_t \de x \\
    &\le \int_{\B(0,2R)\setminus\B(0,R)} \frac{\Lambda}{R} \left( b^\varepsilon_t + a^\varepsilon_t \abs{x} \right) \min\left\lbrace e^{W(t)\abs{x}^2}, N\right\rbrace\varrho^\varepsilon_t \de x \\
    & \quad + 2\int \Phi_R^N(x) W(t)\abs{x}\left( b^\varepsilon_t + a^\varepsilon_t \abs{x} \right) \varrho^\varepsilon_t \de x,
\end{align*}
then, putting all together,
\begin{align*}
    &\frac{\de}{\de t}  \int \Phi_R^N(x)  \varrho^\varepsilon_t(x) \de x \\
    & \le \int \underbrace{\left( W'(t) +2\varepsilon W(t)^2 + A W(t) \right)}_{=0}\abs{x}^2 \Phi_R^N(x) \varrho^\varepsilon_t \de x \\
    & \quad +\left(\varepsilon n +2 \max_{t \in [0,1]} b^\varepsilon_t \right) C \int \Phi_R^N(x)  \varrho^\varepsilon_t(x) \de x +\hat{C} \int_{\B(0,2R)\setminus\B(0,R)} \min\left\lbrace e^{W(t)\abs{x}^2}, N\right\rbrace\varrho^\varepsilon_t \de x \\
     & \le \left(\varepsilon n +2 \max_{t \in [0,1]} b^\varepsilon_t \right) C \int \Phi_R^N(x)  \varrho^\varepsilon_t(x) \de x + \hat{C}N\left( 1 - \max_{t \in [0,1]}  \mu^\varepsilon_t\left(\B(0,R) \right) \right).
\end{align*}
By Gr\"onwall's Lemma,
\begin{align*}
    &\int \Phi_R^N(x)  \varrho^\varepsilon_t(x) \de x  \\
    &\le e^{\left(\varepsilon n +2 \underset{t \in [0,1]}{\max}\; b^\varepsilon_t \right) C t}  \int \zeta_R(x) \min\left\lbrace e^{W(0)\abs{x}^2}, N\right\rbrace  \varrho_0(x) \de x \\
    & + \frac{\hat{C}N\left( 1 - \displaystyle \max_{t \in [0,1]} \mu^\varepsilon_t\left(\B(0,R) \right) \right)}{\left(\varepsilon n +2 \underset{t \in [0,1]}{\max}\; b^\varepsilon_t \right) C} \left( e^{\left(\varepsilon n +2 \displaystyle\max_{t \in [0,1]} b^\varepsilon_t \right) C t} -1\right).
\end{align*}
Passing to the limit as $R \to +\infty$, according to Fatou's Lemma and using the continuity of $\displaystyle \max_{t \in [0,1]} \varrho^\varepsilon_t$, we get
\[\int \min\left\lbrace e^{W(t)\abs{x}^2}, N\right\rbrace  \varrho^\varepsilon_t(x) \de x  \le  e^{\left(\varepsilon n +2 \underset{t \in [0,1]}{\max}\; b^\varepsilon_t \right) C_2 t}  \int e^{W(0)\abs{x}^2} \varrho_0(x) \de x. \]
Passing now to the limit as $N \to +\infty$, 
\[\int e^{W(t)\abs{x}^2}  \varrho^\varepsilon_t(x) \de x  \le  e^{\left(\varepsilon n +2 \underset{t \in [0,1]}{\max}\; b^\varepsilon_t \right) C_2 t}  \int e^{W(0)\abs{x}^2} \varrho_0(x) \de x. \]
As a consequence, 
\[\int e^{W(t)\abs{x}^2}  \varrho^\varepsilon_t(x) \de x  < +\infty,\]
then $I < +\infty$ and, using monotonicity of exponential and Markov's inequality, for every $R>0$,
\begin{align*}
    \int_{\R^n \setminus \B(0,R)} \varrho^\varepsilon_t \de\mathcal{L}^n & = \mu^\varepsilon_t\left( \abs{x}\ge R\right) = \mu^\varepsilon_t\left( e^{W(t)\abs{x}^2} \ge e^{W(t)R^2}\right) \\
    &\le \frac{\int e^{W(t)\abs{x}^2}  \varrho^\varepsilon_t(x) \de x}{e^{W(t)R^2}} \le I e^{-W(1)R^2}.
\end{align*}
The desired inequality simply follows passing to the supremum with respect to $t$.
\par\noindent
$(c)$ Consider $k \in \N$. If $k = 0$, the result is trivial. Let us, then, consider only the case $k \ge 1$. By Fubini--Tonelli's Theorem, combined with $(b)$, we get
% \begin{align*}
% \int \abs{x}^k \varrho^\varepsilon_t(x) \de x &= \int_{\overline{\B(0,1)}} \abs{x}^k \varrho^\varepsilon_t(x) \de x + \int_{\R^n \setminus \overline{\B(0,1)}}\abs{x}^k \varrho^\varepsilon_t(x) \de x \le \\
% &\le \tilde{C} + \int_{\R^n \setminus \overline{\B(0,1)}}\left(1 + \int_1^{\abs{x}} kr^{k-1}dr\right) \varrho^\varepsilon_t(x) \de x = \\
% &= \tilde{C} + \int_1^{+\infty} kr^{k-1} \left( \int_{\R^n \setminus \overline{\B(0,r)}}\varrho^\varepsilon_t(x) \de x \right) dr
% \end{align*}
% so, using $(b)$, we get 
\[\int \abs{x}^k \varrho^\varepsilon_t(x) \de x \le \tilde{C} + k \int_1^{+\infty} r^{k-1}e^{W(1)r^2}\de r < +\infty \]
and, passing to the supremum with respect to $t$,
\[\underset{t \in [0,1]}{\sup}\; \int \abs{x}^k \varrho^\varepsilon_t(x) \de x  < +\infty.\]

Now, given $\alpha_0 >0$, take 
\[\alpha(t) = \alpha_0 e^{-Q t}\]
where 
\[Q = \frac{9}{2}\varepsilon\Lambda + \left( \frac{1}{2} \Lambda + \frac{3}{2}\right) \max \left\{ \max_{t \in [0,1]} b^\varepsilon_t, \max_{t \in [0,1]} a^\varepsilon_t \right\}.\]
A direct computation provides
\[\begin{cases}
    \alpha'(t) + Q \alpha(t) = 0,\\
    \alpha(0) = \alpha_0
\end{cases}\]
and $ 0 < \alpha(1) \le \alpha < \alpha_0$. Now, being sufficient to study the case $t \in \left]0,1\right[$, given $R >1$ and $i=1,\dots n$, using the derivation under the integral sign and the parabolic Kato's inequality,
\begin{align*}
    &\frac{\de}{\de t} \int \zeta_R(x)\alpha(t)(1+\abs{x}^2) \abs{\nabla_{x_i}\varrho^\varepsilon_t(x)}\de x = \int \zeta_R(x)\alpha'(t)(1+\abs{x}^2) \abs{\nabla_{x_i}\varrho^\varepsilon_t(x)}\de x\\
    &+ \int \zeta_R(x)\alpha(t)(1+\abs{x}^2) \frac{\partial}{\partial t} \abs{\nabla_{x_i}\varrho^\varepsilon_t(x)}\de x \le \int \zeta_R(x)\alpha'(t)(1+\abs{x}^2) \abs{\nabla_{x_i}\varrho^\varepsilon_t(x)}\de x\\
    &+ \frac{\varepsilon}{2}\int \zeta_R(x)\alpha(t)(1+\abs{x}^2) \Delta \abs{\nabla_{x_i}\varrho^\varepsilon_t(x)}\de x \\
    &-\int \zeta_R(x)\alpha(t)(1+\abs{x}^2) \textup{sgn}(\nabla_{x_i}\varrho^\varepsilon_t(x))\div  \Big( \nabla\psi^\varepsilon_t(x) \nabla_{x_i}\varrho^\varepsilon_t(x) + \nabla_{x_i}\nabla\psi^\varepsilon_t(x) \varrho^\varepsilon_t(x)\Big) \de x.
\end{align*}
Now, using Gauss--Green's formula,
\begin{align*}
    &\int \zeta_R(x)\alpha(t)(1+\abs{x}^2) \Delta \abs{\nabla_{x_i}\varrho^\varepsilon_t(x)}\de x  \\
    &\le \int 9\Lambda\alpha(t) \abs{\nabla_{x_i}\varrho^\varepsilon_t(x)}\de x + 2n\int \zeta_R(x) \alpha(t) \abs{\nabla_{x_i}\varrho^\varepsilon_t(x)}\de x
\end{align*}
% \begin{align*}
%     &\int \zeta_R(x)\alpha(t)(1+\abs{x}^2) \Delta \abs{\nabla_{x_i}\varrho^\varepsilon_t(x)}\de x = \int \Delta \zeta_R(x) \alpha(t)(1+\abs{x}^2) \abs{\nabla_{x_i}\varrho^\varepsilon_t(x)}\de x + \\
%     &+4 \int \nabla \zeta_R(x) \cdot x \alpha(t)\abs{\nabla_{x_i}\varrho^\varepsilon_t(x)}\de x + 2n\int \zeta_R(x) \alpha(t) \abs{\nabla_{x_i}\varrho^\varepsilon_t(x)}\de x \le \\
%     &\le \int 9\Lambda\alpha(t) \abs{\nabla_{x_i}\varrho^\varepsilon_t(x)}\de x + 2n\int \zeta_R(x) \alpha(t) \abs{\nabla_{x_i}\varrho^\varepsilon_t(x)}\de x
% \end{align*}
and
\begin{align*}
-\int & \zeta_R(x)\alpha(t)(1+\abs{x}^2) \textup{sgn}(\nabla_{x_i}\varrho^\varepsilon_t(x)) \div \Big( \nabla\psi^\varepsilon_t(x) \nabla_{x_i}\varrho^\varepsilon_t(x) + \nabla_{x_i}\nabla\psi^\varepsilon_t(x) \varrho^\varepsilon_t(x)\Big) \de x \\
& \le \left(\frac{1}{2} \Lambda + \frac32\right)\max \left\{ \max_{t \in [0,1]} b^\varepsilon_t, \max_{t \in [0,1]} a^\varepsilon_t \right\}  \int  \alpha(t)(1+\abs{x}^2) \abs{\nabla_{x_i}\varrho^\varepsilon_t(x)}\de x + \tilde{C}.
\end{align*}
From this, it comes that 
\[\frac{\de}{\de t} \int \zeta_R(x)\alpha(t)(1+\abs{x}^2) \abs{\nabla_{x_i}\varrho^\varepsilon_t(x)}\de x \le n\varepsilon\int \zeta_R(x) \alpha(t) \abs{\nabla_{x_i}\varrho^\varepsilon_t(x)}\de x + \tilde{C}\]
so, by Gr\"onwall's Lemma and Fatou's Lemma,
\[\underset{t \in [0,1]}{\sup}\; \int \left(1+\abs{x}^2\right)\abs{\nabla_{x_i}\varrho^\varepsilon_t(x)} \de x  < +\infty.\]
It is now a direct consequence of the definition of gradient to obtain
\[\underset{t \in [0,1]}{\sup}\; \int \left(1+\abs{x}^2\right)\abs{\nabla\varrho^\varepsilon_t(x)} \de x  < +\infty.\]

The last part of the statement can be obtained in an analogous way.
\end{proof}

The last technical Lemma is a density result for the class of test function used in the dual problem for the Benamou--Brenier formulation.

\begin{lemma}\label{density}
Let $\mu_0$ and $\mu_1$ satisfying Assumptions \ref{assumptions}. Then, for every $\varepsilon >0$,
\begin{align*}
    &\sup \left\lbrace \int f_1 \de\mu_1 - \int f_0 \de \mu_0: f \in C^1([0,1])\cap C_b^2(\R^n), \frac{\partial}{\partial t} f + \frac{\varepsilon}{2}\Delta f + \frac{1}{2}\abs{\nabla f}^2 \le 0 \right\rbrace \\
    & = \sup \left\lbrace \int f_1 \de\mu_1 - \int f_0 \de \mu_0: f \in \mathfrak{H}, \frac{\partial}{\partial t} f + \frac{\varepsilon}{2}\Delta f + \frac{1}{2}\abs{\nabla f}^2 \le 0\right\rbrace 
\end{align*}
where 
\[\mathfrak{H} = \left\lbrace f \in C^1([0,1])\cap C^2(\R^n):
    m^\varepsilon_t \abs{\xi}^2 \le \xi \cdot D^2f(t,x) \xi \le M^\varepsilon_t\abs{\xi}^2, m^\varepsilon_t,M^\varepsilon_t \in \R  \right\rbrace\]
and $m^\varepsilon_t,M^\varepsilon_t \in \R$ depend on the marginals as in Lemma \ref{psi_t} (b).
\end{lemma}
\begin{proof}
First of all, the inequality
\begin{align*}
    &\sup \left\lbrace\int f_1 \de\mu_1 - \int f_0 \de \mu_0: f \in C^1([0,1])\cap C_b^2(\R^n), \frac{\partial}{\partial t} f + \frac{\varepsilon}{2}\Delta f + \frac{1}{2}\abs{\nabla f}^2 \le 0 \right\rbrace \\
    & \le \sup \left\lbrace \int f_1 \de\mu_1 - \int f_0 \de \mu_0: f \in \mathfrak{H}, \frac{\partial}{\partial t} f + \frac{\varepsilon}{2}\Delta f + \frac{1}{2}\abs{\nabla f}^2 \le 0\right\rbrace 
\end{align*}
is trivial. For the converse inequality, take $f \in \mathfrak{H}$ such that 
\[\frac{\partial}{\partial t} f + \frac{\varepsilon}{2}\Delta f + \frac{1}{2}\abs{\nabla f}^2 \le 0\]
and consider, given $R>1$,
\[f_R(t,x)= \varepsilon \log \left( \int \left(e^{\frac{f(1,y)}{\varepsilon}}\zeta_R(y) +\frac{1}{R}\right)\mathsf{r}_{\varepsilon(1-t)}(x,y) \de y\right). \]
By construction, $f_R \in C^1([0,1])\cap C_b^2(\R^n)$ and 
\[\frac{\partial}{\partial t} f_R + \frac{\varepsilon}{2}\Delta f_R + \frac{1}{2}\abs{\nabla f_R}^2 = 0,\]
so the result follows by the dominated convergence Theorem.
\end{proof}

\subsection{Proof of the main Theorem}\label{proof_main_theorem}
We are now able to prove the equivalence between the optimal relative entropy and the kinetic energy provided by the Fokker-Planck pair $(\mu^\varepsilon_t,\nabla\psi^\varepsilon_t)$ defined in \eqref{eq:rho_mu} and \eqref{eq:psi}. The proof relies mostly on a derivation under the integral sign and on an application of Gauss--Green's formula.
\begin{proposition}\label{fluidnotvar}
Let $\mu_0$ and $\mu_1$ satisfying Assumptions \ref{assumptions}. Then, for every $\varepsilon >0$,
\[  \varepsilon H(\gamma^\varepsilon \mid \mathsf{R}_\varepsilon ) = \varepsilon H(\mu_0 \mid \mathcal{L}^n) +\int \int_0^1 \frac{\abs{\nabla\psi_t^\varepsilon}^2}{2}\de\mu^\varepsilon_t \de t.\]
\end{proposition}
\begin{proof}
First of all, using the marginal constraints on $\gamma^\varepsilon$,
\[ 
H(\gamma^\varepsilon \mid \mathsf{R}_\varepsilon) = \int \log (f^\varepsilon \otimes g^\varepsilon) \de \gamma^\varepsilon = \int \log f^\varepsilon \de\mu_0 + \int \log g^\varepsilon \de\mu_1.
\]
Now, by Lemma \ref{psi_t} and Proposition \ref{varrho_t}, for every $t \in [0,1]$, we can argue that the function $ x \mapsto \psi^\varepsilon_t(x) \varrho^\varepsilon_t(x)$ is integrable and that, for every $x \in \R$, the function $t \mapsto \psi^\varepsilon_t(x) \varrho^\varepsilon_t(x)$ is differentiable, with $\abs*{\partial_t \left( \psi^\varepsilon_t(x) \varrho^\varepsilon_t(x) \right)}$ dominated by an integrable function, uniformly in $t$.
We can thus differentiate under the integral sign, finding
\[\frac{\de}{\de t} \int  \psi^\varepsilon_t \de\mu^\varepsilon_t = \int \frac{\partial}{\partial t} (\psi^\varepsilon_t \varrho^\varepsilon_t) \de\mathcal{L}^n = \int \left( \frac{\partial}{\partial t}\psi^\varepsilon_t \varrho^\varepsilon_t + \psi^\varepsilon_t \frac{\partial}{\partial t} \varrho^\varepsilon_t \right) \de\mathcal{L}^n, \]
but, using the equations that $\psi_t$ and $\varrho_t$ solve,
\begin{align*}
\frac{\de}{\de t} \int \psi^\varepsilon_t \de\mu^\varepsilon_t = &-\frac{\varepsilon}{2}\int \Delta \psi^\varepsilon_t \varrho^\varepsilon_t \de\mathcal{L}^n -\frac{1}{2}\int \abs{\nabla \psi^\varepsilon_t}^2 \varrho^\varepsilon_t \de\mathcal{L}^n \\
&-\frac{\varepsilon}{2} \int \psi^\varepsilon_t \Delta \varrho^\varepsilon_t \de\mathcal{L}^n - \int \psi^\varepsilon_t \div (\varrho^\varepsilon_t \nabla \psi^\varepsilon_t) \de\mathcal{L}^n.
\end{align*}
Then, according to Gauss--Green's Formula, thanks to Lemma \ref{psi_t} and Proposition \ref{varrho_t},
\[\frac{\de}{\de t} \int \psi^\varepsilon_t \de\mu^\varepsilon_t = \int \frac{\abs{\nabla \psi^\varepsilon_t}^2}{2}\varrho^\varepsilon_t \de\mathcal{L}^n.\]
Integrating both sides with respect to $t$ we find,
\[\int \psi^\varepsilon_1 \de\mu_1 - \int \psi^\varepsilon_0 \de\mu_0 = \int_0^1\int \frac{\abs{\nabla \psi^\varepsilon_t}^2}{2}\varrho^\varepsilon_t \de\mathcal{L}^n \de t. \]
Noticing that $\psi^\varepsilon_0 = \varepsilon\log \varrho_0 - \varepsilon\log f^\varepsilon$ and $\psi^\varepsilon_1 = \varepsilon\log g^\varepsilon$, the result follows.
\end{proof}
Thanks to the previous Proposition, we get one inequality for our main Theorem. The other is obtained by duality.

\begin{proof}[Proof of Theorem \ref{main_theorem}.]
Consider $(\eta_t\mathcal{L}^n,v_t)$ a Fokker--Planck pair such that $\eta_0 = \varrho_0, \eta_1 = \varrho_1$ and $f \in C^1([0,1])\cap C_b^2(\R^n)$ such that
\[\frac{\partial}{\partial t} f + \frac{\varepsilon}{2}\Delta f + \frac{1}{2}\abs{\nabla f}^2 \le 0\]
We have, thanks to Remark \ref{weakeningtestfunctions},
\begin{align*}
    &\int_0^1 \int \frac{\abs{v_t}^2}{2}\eta_t \de\mathcal{L}^n \de t  = \int_0^1 \int \frac{\abs{v_t}^2}{2}\eta_t \de\mathcal{L}^n \de t \\
    &-\int_{0}^{1} \int \left( \frac{\partial}{\partial t} f + \frac{\varepsilon}{2}\Delta f + v_t\cdot \nabla f \right) \eta_t \de\mathcal{L}^n \de t + \int f_1 \eta_1 \de\mathcal{L}^n - \int f_0 \eta_0 \de\mathcal{L}^n \\
    & \ge - \int_0^1 \int \left( \frac{\partial}{\partial t} f + \frac{\varepsilon}{2}\Delta f + \frac{1}{2}\abs{\nabla f}^2  \right)\eta_t \de\mathcal{L}^n \de t + \int f_1 \de\mu_1 - \int f_0 \de\mu_0 \\
    & \ge \int f_1 \de\mu_1 - \int f_0 \de\mu_0.
\end{align*}
Passing to the supremum with respect to $f \in C^1([0,1])\cap C_b^2(\R^n)$ and to the infimum with respect to $(\eta_t\mathcal{L}^n,v_t)$,
\begin{align*}
    &\inf \left\lbrace \int_0^1 \int \frac{\abs{v_t}^2}{2}\eta_t \de\mathcal{L}^n \de t: (\eta_t\mathcal{L}^n,v_t) \textup{  Fokker--Planck pair, } \eta_0 = \varrho_0, \eta_1 = \varrho_1\right\rbrace \\
    & \ge \sup \left\lbrace \int f_1 \de\mu_1 - \int f_0 \de \mu_0: f \in C^1([0,1])\cap C_b^2(\R^n), \frac{\partial}{\partial t} f + \frac{\varepsilon}{2}\Delta f + \frac{1}{2}\abs{\nabla f}^2 \le 0 \right\rbrace. 
\end{align*}
By Lemma \ref{density}, $\psi^\varepsilon_t$ is an admissible competitor in the maximization problem, namely
\begin{align*}
    &\inf \left\lbrace \int_0^1 \int \frac{\abs{v_t}^2}{2}\eta_t \de\mathcal{L}^n \de t: (\eta_t\mathcal{L}^n,v_t) \textup{  Fokker--Planck pair, } \eta_0 = \varrho_0, \eta_1 = \varrho_1\right\rbrace \\
    & \ge \int \psi^\varepsilon_1 \de\mu_1 - \int \psi^\varepsilon_0 \de\mu_0.
\end{align*}
Observing that $\psi^\varepsilon_0 = \varepsilon\log \varrho_0 - \varepsilon\log f^\varepsilon$ and $\psi^\varepsilon_1 = \varepsilon\log g^\varepsilon$, we get, by Propositions \ref{Setting} and \ref{fluidnotvar},
\begin{align*}
    &\inf \left\lbrace \int_0^1 \int \frac{\abs{v_t}^2}{2}\eta_t \de\mathcal{L}^n \de t: (\eta_t\mathcal{L}^n,v_t) \textup{  Fokker--Planck pair, } \eta_0 = \varrho_0, \eta_1 = \varrho_1\right\rbrace \\
    & \ge \int \varepsilon\log g^\varepsilon \de\mu_1 + \int \varepsilon\log f^\varepsilon \de\mu_0 - \varepsilon H(\mu_0 \mid \mathcal{L}^n) = \varepsilon H(\gamma^\varepsilon \mid \mathsf{R}_\varepsilon) - \varepsilon H(\mu_0 \mid \mathcal{L}^n) \\
    & = \varepsilon \underset{\gamma \in \Gamma(\mu_0,\mu_1)}{\min} H(\gamma \mid \mathsf{R}_\varepsilon ) - \varepsilon  H(\mu_0 \mid \mathcal{L}^n) = \int_0^1 \int \frac{\abs{\nabla\psi_t^\varepsilon}^2}{2}\de\mu^\varepsilon_t \de t  \\
    &=\inf \left\lbrace \int_0^1 \int \frac{\abs{v_t}^2}{2}\eta_t \de\mathcal{L}^n \de t: (\eta_t\mathcal{L}^n,v_t) \textup{  Fokker--Planck pair, } \eta_0 = \varrho_0, \eta_1 = \varrho_1\right\rbrace.
\end{align*}
Then, 
\begin{align*}
\varepsilon &\underset{\gamma \in \Gamma(\mu_0,\mu_1)}{\min} H(\gamma \mid \mathsf{R}_\varepsilon ) = \varepsilon H(\mu_0 \mid \mathcal{L}^n) \\
&+\min \left\lbrace \int_0^1 \int \frac{\abs{v_t}^2}{2}\eta_t \de\mathcal{L}^n \de t: (\eta_t\mathcal{L}^n,v_t) \textup{  Fokker--Planck pair, } \eta_0 = \varrho_0, \eta_1 = \varrho_1\right\rbrace.
\end{align*}
As regards the uniqueness, consider the change of variables $q_t = \eta_t v_t$, the set 
\[\mathscr{C} =  \left\lbrace (\eta,q): (\eta_t\mathcal{L}^n, v_t) \textup{ is a Fokker--Planck pair, } \eta_0\mathcal{L}^n = \mu_0, \eta_1\mathcal{L}^n = \mu_1\right\rbrace,\]
the convex and lower semi-continuous function $\Theta: \left[0,+\infty\right[ \times\R \to [0,+\infty]$ such that
\[\Theta(x,y) = \begin{cases}
    \frac{y^2}{x} & \textup{if } x>0,\\
    0 & \textup{if } x=y=0,\\
    +\infty & \textup{otherwise},\\
\end{cases}\]
and the convex functional $\mathfrak{F} : \mathscr{C} \to [0,+\infty]$ such that
\[\mathfrak{F}(\eta,q) = \int_0^1 \int \frac{1}{2}\Theta\left(\eta_t,\abs{q_t}\right) \de \mathcal{L}^n \de t = \int_0^1 \int \frac{\abs{v_t}^2}{2}\eta_t \de\mathcal{L}^n \de t.\]
Take another minimizer $(\Bar{\eta}_t,\Bar{v}_t)$ for kinetic energy minimization and set $\Bar{q}= \Bar{\eta}\Bar{v}$. By minimality, $\mathfrak{F}(\overline{\eta}_t,\overline{q}_t) = \mathfrak{F}(\varrho^\varepsilon_t,\varrho^\varepsilon_t \nabla\psi^\varepsilon_t)$ so, given $\lambda \in \left]0,1\right[$ and setting
\begin{align*}
    \eta^\lambda_t &= (1-\lambda) \eta_t + \lambda \varrho^\varepsilon_t, & q^\lambda_t &= (1-\lambda) q_t + \lambda \varrho^\varepsilon_t \nabla\psi^\varepsilon_t,
\end{align*}
we get 
\[\mathfrak{F}(\eta^\lambda_t,q^\lambda_t) \ge \mathfrak{F}(\varrho^\varepsilon_t,\varrho^\varepsilon_t \nabla\psi^\varepsilon_t) = (1-\lambda) \mathfrak{F}(\overline{\eta}_t,\overline{q}_t) + \lambda \mathfrak{F}(\varrho^\varepsilon_t,\varrho^\varepsilon_t \nabla\psi^\varepsilon_t)\] 
and, using also the convexity of $\mathfrak{F}$ to obtain the converse inequality, we arrive to
\[\mathfrak{F}(\eta^\lambda_t,q^\lambda_t) = (1-\lambda) \mathfrak{F}(\overline{\eta}_t,\overline{q}_t) + \lambda \mathfrak{F}(\varrho^\varepsilon_t,\varrho^\varepsilon_t \nabla\psi^\varepsilon_t),\] 
so $(\eta^\lambda_t,q^\lambda_t)$ also minimizes $\mathfrak{F}$ for every $\lambda \in [0,1]$. That means that $\mathfrak{F}$ is constant on the line connecting $(\varrho^\varepsilon_t,\varrho^\varepsilon_t \nabla\psi^\varepsilon_t)$ and $(\overline{\eta}_t,\overline{q}_t)$ and (since we are in the case $x>0$) this permits us to say that for every $\lambda \in [0,1]$,
\[\Theta(\eta^\lambda_t,q^\lambda_t) = (1-\lambda) \Theta(\overline{\eta}_t,\overline{q}_t) + \Theta(\varrho^\varepsilon_t,\varrho^\varepsilon_t \nabla\psi^\varepsilon_t).\]
But the previous implies that there exists $\alpha(t,x)>0$ such that
\[(\overline{\eta}_t,\overline{q}_t) = \alpha(t,x)(\varrho^\varepsilon_t,\varrho^\varepsilon_t \nabla\psi^\varepsilon_t),\]
so, using the fact that both $(\overline{\eta}_t,\overline{v}_t)$ and $(\varrho^\varepsilon_t,\varrho^\varepsilon_t \nabla\psi^\varepsilon_t)$ are Fokker--Planck pairs, we get, weakly,
\[\frac{\partial}{\partial t} \alpha_t + \nabla \alpha_t \cdot \nabla \psi^\varepsilon_t = 0.\]
By the fact that both $\overline{\eta}_t$ and $\varrho^\varepsilon_t$ are probability measures, we get that $\alpha_0, \alpha_1 = 1$ a.e. in $\R^n$. In other words, $\alpha$ is a weak solution of
\[\begin{cases}
    \frac{\partial}{\partial t} \alpha_t + \nabla \alpha_t \cdot \nabla \psi^\varepsilon_t = 0 & \textup{in } \left]0,1\right[ \times \R^n,\\
    \alpha_0 = 1 & \textup{on } \R^n,\\
    \alpha_1 = 1 & \textup{on } \R^n,\\
\end{cases}\]
and, since the characteristics method implies the uniqueness of the solution, we get $\alpha = 1$. It is now sufficient to use the strict convexity of the map $v_t \mapsto \int_0^1\int \frac{\abs{v_t}^2}{2}\varrho^\varepsilon_t \de\mathcal{L}^n \de t$
to conclude. 
\end{proof}

\paragraph*{Acknowledgment} L.N. and M.G. benefited from the support of the FMJH Program PGMO and from the ANR project GOTA (ANR-23-CE46-0001). While this work was written, L.T.\ was associated to INdAM (Istituto Nazionale di Alta Matematica ``Francesco Severi'') and the group GNAMPA.

\par\noindent M. G. is funded by the European Union and, as part of the MSCA program, is also supported, as secondment, by Laboratoire Jean Alexandre Dieudonné, Université Côte d'Azur, CNRS. This project has received funding from the European Union’s MSCA--Horizon Europe, grant agreement No 101126554. Views and opinions expressed are however those of the author(s) only and do not necessarily reflect those of the European Union or the European Research Executive Agency. Neither the European Union nor the granting authority can be held responsible for them.
\begin{center}
    \includesvg[height=0.1\textheight]{EU.svg} 
\end{center}

\par\noindent This work is licensed under Creative Commons Attribution 4.0 International. To view a copy of this license, visit \url{https://creativecommons.org/licenses/by/4.0/}.

%\newpage

\printbibliography[heading = bibintoc]	% Bibliography
\end{document}